\documentclass{scrartcl}
\usepackage{bbm}
\usepackage{enumerate}
\usepackage{mathtools}
\usepackage{amsfonts}
\usepackage{amsthm}
\usepackage{algpseudocode}
\DeclareNewTOC[
type = algorithm ,
float ,
name = Algorithm ,
]{algorithm}

\usepackage{cite}

\newcommand{\ddt}{\tfrac{\mathrm{d}}{\mathrm{d} t}}

\DeclareMathOperator{\IM}{im}
\DeclareMathOperator{\RE}{Re}
\DeclareMathOperator{\diag}{diag}

\def\diff{\text{d}}
\newcommand\ones{\mathbbm{1}}
\newcommand\tr{\top}
\newcommand\varsplit{x}
\newcommand{\real}{\mbox{$\mathbbm{R}$}}
\newcommand{\nats}{\mathbb{N}}

\newcommand\cond{\text{cond}}

\newtheorem{theorem}{Theorem}
\newtheorem{prop}[theorem]{Proposition}%
\newtheorem{example}[theorem]{Example}%
\newtheorem{remark}[theorem]{Remark}%

\newtheorem{ass}[theorem]{Assumption}
\newtheorem{definition}[theorem]{Definition}%
\newtheorem{lemma}[theorem]{Lemma}

\begin{document}

\title{Operator Splitting Based Dynamic Iteration for Linear Port-Hamiltonian Systems}

\author{Andreas Bartel${}^{1}$, Michael Günther${}^{1}$, Birgit Jacob${}^{1}$ and Timo Reis${}^{2}$\\[1ex]
 ${}^{1}$ IMACM, Bergische Unviersität Wuppertal,\\ Gaußstr. 20, 42119 Wuppertal, Germany \\[1ex]
 ${}^{2}$ Institut f\"ur Mathematik, Technische Universit\"at Ilmenau,\\ Weimarer Str. 25, 98693 Ilmenau, Germany}
\maketitle

{\textbf{Abstract.} A dynamic iteration scheme for linear differential-algebraic port-Hamil\-tonian systems based on  Lions-Mercier-type operator splitting methods is developed.  The dynamic iteration is monotone in the sense that the error is decreasing and no stability conditions are required. The developed iteration scheme is even new for linear port-Hamiltonian systems. The obtained algorithm is applied to multibody systems and electrical networks.%
}
\bigskip

{\textbf{Keywords. } Port-Hamiltonian Systems, DAE, Operator Splitting, Dynamic Iteration.}
\bigskip

\textbf{MSC Classification. } 37Jxx, 34A09
\bigskip

\section{Introduction}
Dynamic iteration schemes for differential-algebraic equations (DAE) have been widely used and discussed for multiphysics problems, which allow for an easy exploitation of the different properties of the subystems~\cite{argu2001,ali2015,bagu2018,babs2014,bbgs2012,Arnold2022}. In general, we consider an initial value problem for a (semi-explicit index-1) DAE on $t\in[0,T]$ of the form 
\begin{subequations}
	\label{eq.dae}
	\begin{align}
		\dot v & = f(t,v,w), \quad v(0)=v_{0},\\
		0 & = g(t,v,w),
	\end{align}
\end{subequations}
where $\partial g(t,v,w)/\partial w$ is regular in a neighborhood of the solution $(v(t), w(t))^\tr$, $t\in [0, T]$. 
A general dynamic iteration scheme is given by an initial guess
$(v^{(0)}(t), w^{(0)}(t))^\tr$ ($t\in[0,T]$, i.e., an initial waveform) and the solution $(v^{(k)}(t), w^{(k)}(t))^\tr$ of the DAE for integer $k\ge 1$
\begin{align*}
	\dot v^{(k)} & = F(t,v^{(k)},v^{(k-1)},w^{(k)},w^{(k-1)}), \quad v^{(k)}(0)=v_0,\\
	0 & = G(t,v^{(k)},v^{(k-1)},w^{(k)},w^{(k-1)}) 
\end{align*}
with arbitrary splitting functions $F,G$ fulfilling the %
compatibility conditions 
\[ 
F(t,v,v,w,w)=f(t,v,w) \quad \text{and} \quad G(t,v,v,w,w)=g(t,v,w). 
\]
Hence a dynamic iteration defines as a mapping from
the $(k-1)$-th iterate to the $k$-th iterate: 
\[ (v^{(k-1)}(t),\,w^{(k-1)}(t)) \mapsto    
(v^{(k)}(t),w^{(k)}(t)).
\]
One drawback of such schemes for DAE is given by the fact that stability conditions for a guaranteed convergence must hold in a twofold manner: 
\begin{enumerate}[(a)]
	\item %
	convergence within one time window has to be present~\cite{argu2001,jaka1996}, and 
	\item %
	stable error propagation from window to window is needed~\cite{argu2001,ali2015}. 
\end{enumerate}
If one of these conditions is not met choosing arbitrary small time windows will not help out (in general).

In the case of initial value problems of ordinary differential equations
\begin{align*}
	\dot x & = f(t,x), \quad x(0)=x_0, \quad t\in [0,T],
\end{align*}
convergence of dynamic iteration schemes 
(analogously defined) 
is given 
for arbitrary long time windows $[0,T]$. However, the convergence might not be monotone, and the error might increase drastically in the beginning, causing an overflow in computer implementations, 
see chapter~\ref{chapter.failure} for an example. 

In this work, we show that an alternative iterative approach for linear DAE systems can avoid both flaws described above if the DAE system~\eqref{eq.dae} is composed of coupled port-Hamiltonian systems. This type of systems is motivated by physics, provides an energy balance and gives a~simple framework for coupled systems \cite{9030180,JvdS14,MascvdSc18,MvdS20,vdS04,vdS13,GeHaRe20,GeHaRevdS20,BMXZ18}. Now, the alternative iterative approach is based on an iteration scheme from Lions and Mercier~\cite{Lions1979}. The convergence will be monotone, and no stability conditions appear.

The paper is organized as follows: in the subsequent  chapter~2, the framework of linear coupled port-Hamiltonian DAE systems is set. 
Both, the perspective of an overall coupled system and the perspective of a system composed of coupled port-Hamiltonian DAE subsystems is given. Chapter 3 illustrates the limits of classical dynamic iteration schemes such as the Jacobi iteration, which motivates the derivation and analysis of operator splitting based dynamic iteration schemes in the following. Chapter 4 recapitulates some basic facts on maximal monotone operators, which are used in chapter 5 to derive the monotone convergence results of operator splitting based dynamic iteration schemes. An estimate for the convergence rate can be given in the case of ODE systems and DAE systems with enough dissipation. 
Numerical results are discussed in chapter 6. The paper finishes with some concluding remarks and an outlook.

\section{The setting}\label{sec:setting}
We consider linear time-invariant initial value problems of DAEs, which have the following structure:

\begin{definition}
	For a final time $T>0$, the descriptor variable $x(t) \in \real^n$ and $t\in [0,T]$, we consider
	\begin{subequations}\label{eqn:DAEsystemQ}
		\begin{align}
			\ddt E x(t)&=(J-R)Qx(t)+Bu(t),\quad t\in [0,T],  \quad Ex(0)=Ex_0,\\
			y(t)&=B^\tr Qx(t),
			%
		\end{align}
	\end{subequations}
	where $E$, $J$, $R$ and $Q$ are real $n\times n$-matrices, $B$ is a $n\times m$-matrix, $E x_0 \in \real^n$ is a given initial value, $u$ is a given input and $y$ is referred to as output. (The matrix $E$ may have no full rank.) 
	\hfill $\Box$
\end{definition}

Throughout this article, we are working with weak solutions. 

\begin{definition}[Solution]
	Let $x_0\in \real ^n$ and $u\in L^2([0,T], \real^m)$.
	A function $x\in L^2([0,T],\real^n)$ is called \textrm{(weak) solution} of \eqref{eqn:DAEsystemQ} if $Ex \in H^1([0,T],\real^n)$, $Ex(0)=Ex_0$ and \eqref{eqn:DAEsystemQ} is satisfied 
	for $t\in[0,T]$ almost everywhere. \hfill $\Box$
\end{definition}

\begin{prop}\label{prop:PHS-structure}
	Assume that a~differential-algebraic system \eqref{eqn:DAEsystemQ} is given with $E,J,R\in\real^{n\times n}$, $B\in\real^{n\times m}$ satisfying $E^\tr Q=Q^\tr E\ge0$, $R=R^\tr\ge 0$ and $J^\tr=-J$. Furthermore, let $u\in L^2([0,T],\real^m)$, $x_0\in\real^n$, and let $x\in L^2([0,T],\real^n)$ be a solution of \eqref{eqn:DAEsystemQ}. Then the following statements hold:
	\begin{enumerate}[(a)]
		\item The function $t\mapsto x(t)^\tr Q^\tr Ex(t)$ is weakly differentiable with
		\[
		\ddt \big(x^\tr Q^\tr Ex\big)
		= 2x^\tr Q^\tr\big(\ddt Ex) \,\in L^1([0,T],\real).
		\]
		\item The following {\em dissipation inequality} holds for all $t\in[0,T]$:
		\begin{multline*}
			\tfrac12\big(x(t)^\tr Q^\tr Ex(t)\big)-\tfrac12\big(x_0^\tr Q^\tr Ex_0\big)
			\\
			=-\int_0^tx(\tau)^\tr Q^\tr RQ x(\tau)d\tau
			+\int_0^tu(\tau)^\tr y(\tau)d\tau
			\leq \int_0^tu(\tau)^\tr y(\tau)d\tau
			.
		\end{multline*}
	\end{enumerate}
\end{prop}

\begin{proof}
	\begin{enumerate}[(a)]
		\item We have $E^\tr Q=Q^\tr E$, and thus, for the Moore-Penrose inverse $E^+$ of $E$, we have \[E^\tr Q=E^\tr (E^+)^\tr E^\tr Q=E^\tr (E^+)^\tr Q^\tr E.\] Consequently, for $X=(E^+)^\tr Q^\tr\in \mathbb{R}^{n\times n}$, we have $Q^\tr E=E^\tr XE$. By using that $Ex \in H^1([0,T],\mathbb R^n)$, the product rule for weak derivatives \cite[Thm.~4.25]{Alt16} implies that
		\[x^\tr Q^\tr Ex=(Ex)^\tr X (Ex)\in W^{1,1}([0,T],\mathbb{R})\]
		with
		\begin{align*}
			\ddt(x^\tr Q^\tr Ex) & = \ddt \left((Ex)^\tr X (Ex) \right)
			= 2  x^\tr E^\tr X \ddt\left(Ex\right)
			= 2  x^\tr \ddt\left( E^\tr X Ex\right) \\
			&= 2  x^\tr \ddt\left( Q^\tr Ex\right)
			= 2  x^\tr Q^\tr\ddt\left(  Ex\right).
		\end{align*}
		\item The previous statement yields
		\begin{align*}
			\ddt\tfrac12(x^\tr Q^\tr Ex)
			&=x^\tr Q^\tr \ddt(Ex)
			= x^\tr Q^\tr ((J-R)Qx+Bu) 
			\\
			&
			= -(Qx)^\tr R (Qx) + x^\tr Q Bu \le (B^\tr Qx)^\tr u = y^\tr u.
		\end{align*}
		Now an integration on $[0,t]$ leads to the dissipation inequality.
	\end{enumerate}
	\vspace*{-.9\baselineskip}
\end{proof}

Thus, we investigate the following class of port-Hamiltonian DAEs: 
\begin{definition}[PH-DAE]\label{def:PH-DAE} 	
The system~\eqref{eqn:DAEsystemQ} with the assumptions of Prop.~\ref{prop:PHS-structure} is here referred  as port-Hamiltonian DAE (PH-DAE).
	\hfill $\Box$
\end{definition}

\subsection{Perspective as overall coupled system}
We further assume that \eqref{eqn:DAEsystemQ} is composed of several port-Hamiltonian systems which are coupled in an energy-preserving way. The assumptions are collected in the following.
\begin{ass}\label{ass:overall-ivp}(Overall coupled PH-DAE) We consider the setting \eqref{eqn:DAEsystemQ}.
	\begin{enumerate}[(a)]
		\item\label{ass:overall-ivpa} The matrices in \eqref{eqn:DAEsystemQ} are, for some $s,n_1,\ldots,n_s\in\mathbb{N}$, structured as
		\begin{gather}
			E=\begin{bmatrix}E_1&&\\&\ddots\\&&E_s\end{bmatrix}\!,
			\;\; 
			R= \begin{bmatrix}R_1&&\\&\ddots\\&&R_s\end{bmatrix}\!,
			\;\;
			Q=\begin{bmatrix}Q_1&&\\&\ddots\\&&Q_s\end{bmatrix}\!,\label{eq:EQBprop}
			\\
			J=\begin{bmatrix}
				J_1&J_{12}&\cdots&J_{1s}\\
				-J_{12}^\tr&\ddots&\ddots&\vdots\\
				\vdots&\ddots&\ddots&J_{s-1,s}\\
				-J_{1s}^\tr&\cdots&-J_{s-1,s}^\tr &J_s
			\end{bmatrix}\!,\;\;
			B=\begin{bmatrix}B_1\\\vdots\\\vdots\\B_s\end{bmatrix},\label{eq:EQBprop2} %
		\end{gather}
		with $n_1+n_2+\dotsc + n_s =n$, where 
		\begin{enumerate}[(i)]
			\item for $i=1,\dotsc,\, s$, $E_i,Q_i,J_i,R_i \in \mathbb{R}^{n_i\times n_i}$ with $R_i=R_i^\tr\ge0$, $J_i^\tr=-J_i$, $E_i^\tr Q_i=Q_i^\tr E_i\ge0$, $Q_i$ invertible and $B_i \in \mathbb R^{n_i\times m}$;
			\item  $J_{ij}\in\mathbb R^{n_i\times n_j}$ for $i,j\in\{1,\ldots,s\}$ with $i<j$.
		\end{enumerate}
		\item\label{ass:overall-ivpb} 
		$\mathrm{rk}\begin{bmatrix}
			E &
			R &
			J
		\end{bmatrix}=n$.
		\item\label{ass:overall-ivpc} $T>0$, $x_0\in\mathbb{R}^n$ and $u\in L^2([0,T];\mathbb{R}^m)$. Furthermore, for matrices $Z\in\mathbb{R}^{n\times r}$, $Z_1\in\mathbb{R}^{r\times r_1}$, with full column rank and
		\begin{align*}
			\IM Z=&\,\ker E,\\
			\IM Z_1=&\,\ker RQZ\cap \ker Z^\tr Q^\tr JQZ,
		\end{align*}
		the input $u$ and the initial value $x_0$  fulfill
		\begin{align*}
			Z_1^\tr Z^\tr Q^\tr Bu\in &\,H^1([0,T];\mathbb{R}^{r_1}) \quad 
			\text{ with }  Z_1^\tr Z^\tr Q^\tr (JQx_0+Bu(0))=0.
			\tag*{$\Box$}
		\end{align*}
		\end{enumerate}
\end{ass}

\begin{remark}\label{remark:partitioned-PHS}
	Notice that with Assumption~\ref{ass:overall-ivp}, we have for \eqref{eqn:DAEsystemQ}:
	\begin{enumerate}[(a)]
		\item  $Q^\tr E=E^\tr Q\geq0,\quad  R=R^\tr\geq 0,\quad \mathrm{and} \quad 
		Q$ 	is invertible. Hence, the overall coupled PH-DAE of Assumption~\ref{ass:overall-ivp} is a PH-DAE (in the sense of Definition~\ref{def:PH-DAE}).
		\item As $Q^\tr E\geq0$, we have that $EQ^{-1}=Q^{-\tr}(Q^TE)Q^{-1}$ is positive semi-definite as well, and therefore possesses a~matrix square root $(EQ^{-1})^{1/2}$. 
		\item \label{remark:partitioned-PHS-subsystems} 
		The $i$th subsystem for the variable $x_i$ reads:
		\[
		E_i x_i = (J_i-R_i)Q_i x_i + B_i u \; - \sum_{j<i} J_{ji}^\tr Q_j x_j 
		+ \sum_{i<j} J_{ij} Q_j x_j.
		\]
		The coupling is facilitated by the off-diagonal terms within the matrix $J$.
		\hfill $\Box$
	\end{enumerate}
\end{remark}

\begin{remark}[Coupled linear PH-DAE systems]
\label{rem.coupledphs}
A PH-DAE system of type~\eqref{eqn:DAEsystemQ} with coupling structure~\eqref{eq:EQBprop} and \eqref{eq:EQBprop2} is naturally given in the case of $s$ coupled linear PH-DAE systems as follows: consider $s$ PH-DAE subsystems 
$$
\begin{aligned}
\ddt E_i  x_i(t) 
    =& (\tilde J_i  - R_i) Q_i x_i(t) + B_i u_i (t),
&y_i(t) =    B_i^\tr  Q_i x_i(t),
 \end{aligned}
$$
$i=1,\ldots,s.$ with $\tilde J_i=-\tilde J_i^\top$, $R_i=R_i^\top \ge 0$, %
and $Q_i^\top E_i=E_i^\top Q_i \ge 0$.
The input $u_i$, the output $y_i$ and $B_i$ are split into
$$
    \label{eq:split-input-output}
    u_i(t)= \begin{pmatrix}
                \hat u_i(t) \\
                \bar u_i(t)
        \end{pmatrix},
        \quad
    y_i(t)= \begin{pmatrix}
                \hat y_i(t) \\
                \bar y_i(t)
        \end{pmatrix}, \quad 
      B_i = \begin{pmatrix}
                \hat B_i &
                \bar B_i
            \end{pmatrix}   
$$
according to external and coupling quantities.
The subsystems are coupled via external inputs and outputs by %
\begin{align*}
    \begin{pmatrix}
    \hat u_1 \\ \vdots \\ \hat u_k
    \end{pmatrix} 
    + \hat C 
    \begin{pmatrix}
    \hat y_1 \\ \vdots \\ \hat y_k
    \end{pmatrix} = 0, \qquad \hat C = - \hat C^\top. 
\end{align*}
These $s$ systems can be condensed to one large PH-DAE system~\cite{bagu2018}
\begin{subequations}\label{cond.ph.dae}
\begin{align}
\ddt E  x & =   (J -R)  Q x  +  \bar B \bar u, \\
\bar y & =  \bar B^{\tr} Q x 
\end{align}
\end{subequations}
with $J=\tilde J - \hat B \hat C \hat B^{\tr}$ and the condensed quantities 
\begin{gather*}
 v^{\tr} =  (v_1^{\tr},\ldots, v_s^{\tr})
    \quad  %
        \text{ for } v \in \{x, \,\bar u, \bar y\},
       \\
  F=\diag\,(F_1,\ldots, F_s)
    \quad %
        \text{ for } F \in \{E,\,Q,\,\tilde J,\,R,\,\hat{B},\, \bar{B}\}.
\end{gather*}
Equation~\eqref{cond.ph.dae}
defines now a PH-DAE system of type~\eqref{eqn:DAEsystemQ} with structure given by  \eqref{eq:EQBprop} and \eqref{eq:EQBprop2}: the matrices $\tilde J$ and $R$ are block-diagonal, and the Schur complement type matrix $-\hat B \hat C \hat B^{\tr}$ has only off-block diagonal entries. Note that the coupling has been shifted from the port matrices $B_i$ to the off-block diagonal part of $J$. \hfill $\Box$
\end{remark}
Now we collect some properties of systems \eqref{eqn:DAEsystemQ} with properties as in 
	Assumption~\ref{ass:overall-ivp}. To this end, we note that a~matrix pencil $sE-A\in\mathbb R[s]^{n\times n}$ is called {\em regualar}, if $\det(sE-A)$ is not the zero polynomial. Furthermore, for a~regular pencil $sE-A$, the {\em index\,}~$\nu\in\mathbb{N}_0$ of the DAE $\ddt E x(t)=Ax(t)+f(t)$ is the smallest number for which the mapping $\lambda\mapsto \lambda^{1-\nu}(\nu E-A)^{-1}$ is bounded outside a~compact subset of $\mathbb{C}$. Note that the index equals to the number of differentiations of the DAE needed to obtain an ordinary differential equation \cite[Chap.~2]{KunkMehr06}.
\begin{prop}\label{prop:DAEbasicprop}
	Any system \eqref{eqn:DAEsystemQ} fulfilling %
	Assumption~\ref{ass:overall-ivp} has the following properties.
	\begin{enumerate}[(a)]
		\item The pencil $sE-(J-R)Q\in\mathbb R[s]^{n\times n}$ is regular.
		\item There exists a~unique solution $x$ of \eqref{eqn:DAEsystemQ}.
			  \item The index of the DAE \eqref{eqn:DAEsystemQ}
	  is at most two.
	\end{enumerate}
\end{prop}
\begin{proof}\
	\begin{enumerate}[(a)]
		\item 
		For $\lambda\in\mathbb C$ with $\RE\lambda >0$, we show that 
		$C:=\lambda E-(J-R)Q$ has a trivial kernel. Let
		$z\in\mathbb C^{n}$ with $Cz=0$. Then
		\[
		0=\RE(z^* Q^\tr C
		z)=\RE(\lambda z^* Q^\tr Ez)-\RE(z^* Q^\tr JQx)+\RE(z^* Q^\tr RQz).
		\]
		Invoking that $J=-J^\tr$,  $Q^\tr E=E^\tr Q\geq0$ and $R=R^\tr \geq0$, we obtain
		\[\begin{aligned}
			\RE(\lambda z^* Q^\tr Ez)=&\,\RE(\lambda)z^* Q^\tr Ez\geq0,\\
			\RE(z^* Q^\tr JQz)=&\,\RE((Qz)^* J(Qz))=0,\\
			\RE(z^* Q^\tr RQz)=&\,\RE((Qz)^* R(Qz))=(Qz)^* R(Qx)\geq0,
		\end{aligned}\]
		which leads to $z^* Q^\tr Ez=0$ and $(Qz)^* R(Qz)=0$. The positive semi-definiteness of $Q^\tr E$ and $R$ implies that $E^\tr Qz=Q^\tr Ez=0$ and $RQz=0$. 
		Using $(\lambda E-(J-R)Q)z=0$, this yields $JQz=0$. 
		Hence, by $R=R^\tr$ and $J=-J^\tr$,
		\[ Qz\in\ker
		\left[\begin{matrix}E&R&J\end{matrix}\right]^\tr.
		\]
		Assumption~\ref{ass:overall-ivp}\,\eqref{ass:overall-ivpb} now gives $Qz=0$. Since $Q$ is invertible, we are led to $z=0$.
		\item By using that the pencil $sE-(J-R)Q$ is regular, uniqueness of solutions of the initial value problem 
		\eqref{eqn:DAEsystemQ} follows from the considerations in \cite[Sec.~1.1]{LamoMarz13}.\\
		It remains to prove that a~solution of \eqref{eqn:DAEsystemQ} exists. 
		Let $Z\in\mathbb{R}^{n\times r}$, $Z_1\in\mathbb{R}^{r\times r_1}$ be as in Assumption~\ref{ass:overall-ivp}\,\eqref{ass:overall-ivpc}. Furthermore, let $Z'\in\mathbb{R}^{n\times (n-r)}$, $Z_1'\in\mathbb{R}^{r\times (r-r_1)}$ be matrices with $\IM Z'=(\IM Z)^\bot$ and $\IM Z_1'=(\IM Z_1)^\bot$. Moreover, we introduce matrices $Z_2\in\mathbb{R}^{(n-r)\times r_1}$, $Z_2'\in\mathbb{R}^{(n-r)\times (n-r-r_1)}$ with full column rank and 
		\[\IM Z_2=\ker Z_1^\tr Z^\tr Q^\tr (J-R)Z',\;\; \IM Z_2'=(\IM Z_2)^\bot.\]
		The construction of $Z,Z',Z_1,Z_1',Z_2,Z_2'$ yields that the columns of the matrix
		\begin{equation}
			V:=\big[Z'Z_2', Z'Z_2,ZZ_1',ZZ_1\big]\label{eq:Tdef}
		\end{equation}
		are linearly independent, and thus, $V,V^\tr Q^\tr\in\mathbb{R}^{n\times n}$ are invertible matrices.
		Then we obtain 
		\begin{equation}V^\tr Q^\tr(sE-(J-R)Q)V=    \begin{bmatrix}
				sE_{11}-A_{11}&sE_{12}-A_{12}&-A_{13}&-A_{14}\\sE_{12}^\tr-A_{21}&sE_{22}-A_{22}&-A_{23}&0\\-A_{31}&-A_{32}&-A_{33}&0\\A_{14}^\tr&0&0&0
			\end{bmatrix},\label{eq:pencrel}
		\end{equation}
		where 
		\begin{align*}
			E_{11}=&\,(Z_2')^\tr(Z')^\tr Q^\tr E Z'Z_2',&E_{12}=&\,(Z_2')^\tr(Z')^\tr Q^\tr E Z'Z_2,\\E_{22}=&\,Z_2^\tr(Z')^\tr Q^\tr E Z'Z_2,&
			A_{11}=&\,(Z_2')^\tr(Z')^\tr Q^\tr (J-R)Q Z'Z_2',\\
			A_{12}=&\,(Z_2')^\tr(Z')^\tr Q^\tr (J-R)Q Z'Z_2,&A_{21}=&\,Z_2^\tr(Z')^\tr Q^\tr (J-R)Q Z'Z_2',\\
			A_{13}=&\,(Z_2')^\tr(Z')^\tr Q^\tr (J-R)Q ZZ_1',&A_{31}=&\,(Z_1)'^\tr Z^\tr Q^\tr (J-R)Q Z'Z_2',\\
			A_{22}=&\,Z_2^\tr(Z')^\tr Q^\tr (J-R)Q Z'Z_2,&A_{33}=&\,(Z_1)'^\tr Z^\tr Q^\tr (J-R)Q ZZ_1',\\
			A_{14}=&\,(Z_2')^\tr(Z')^\tr Q^\tr JQ ZZ_1,\\
			B_{1}=&\,(Z_2')^\tr(Z')^\tr Q^\tr B,&B_{2}=&\,Z_2^\tr (Z')^\tr Q^\tr B,\\
			B_{3}=&\,(Z_1')^\tr Z^\tr Q^\tr B,&B_{4}=&\,Z_1^\tr Z^\tr Q^\tr B.
		\end{align*}
		The construction of the matrices $Z,Z',Z_1,Z_1',Z_2,Z_2'$ yields that
		\begin{itemize}
			\item $\left[\begin{smallmatrix}E_{11}&E_{12}\\E_{12}^\tr&E_{22}\end{smallmatrix}\right]$ is positive definite. In particular, $E_{22}$ is invertible,
			\item $A_{33}$ is invertible, and
			\item $A_{14}$ has full row rank.
		\end{itemize}
		The already proved regularity of $sE-(J-R)Q$, which  implies regularity of the pencil $V^\tr Q^\tr (sE-(J-R)Q)QV$, which yields that $A_{14}$ has moreover full column rank. As a~consequence, $A_{14}$ is square (i.e., $2r_1=n-r$) and invertible.
		\\
		Set  %
		$z_0=(z_{10}^\tr,z_{20}^\tr,z_{30}^\tr,z_{40}^\tr)$, $z_{10}\in\mathbb{R}^{n-r-r_1}$, $z_{20}\in\mathbb{R}^{r_1}$, $z_{30}\in\mathbb{R}^{r-r_1}$, $z_{40}\in\mathbb{R}^{r_1}$. By Assumption~\ref{ass:overall-ivp}\,\eqref{ass:overall-ivpc}, we have $B_4u\in H^1([0,T];\mathbb{R}^{r_1})$ and
		\begin{multline*}
			0= Z_1^\tr Z^\tr Q^\tr (JQx_0+Bu(0))
			=Z_1^\tr Z^\tr Q^\tr JQx_0+B_4u(0)\\
			=Z_1^\tr Z^\tr Q^\tr JQ(Z'Z_2'z_{10}+Z'Z_2z_{20}+ZZ_1'z_{30}+ZZ_1z_{40})+B_4u(0)\\
			=Z_1^\tr Z^\tr Q^\tr JQZ'Z_2'z_{10}+B_4u(0)=-Z_1^\tr Z^\tr Q^\tr J^\tr QZ'Z_2'z_{10}+B_4u(0)\\=-A_{14}^\tr z_{10}+B_4u(0).
		\end{multline*}
		Consequently, $z_1=(A_{14}^\tr)^{-1}B_4u$ fulfills $z_1\in H^1([0,T];\mathbb{R}^r)$, $-A_{14}^\tr z_1+B_4u(t)=0$ and $z_1(0)=z_{10}$. Moreover, let $z_2\in H^1([0,T];\mathbb{R}^{r_1})$ be the solution of the following ODE
		with initial value $z_2(0)=z_{20}$
		\begin{multline*}
			\ddt z_2(t)=E_{22}^{-1}(A_{22}-A_{23}A_{33}^{-1}A_{32})z_2(t)\\
			\quad 
			+E_{22}^{-1}\big(-\ddt E_{12}^\tr z_1(t)+A_{21}z_1(t)+B_2u(t)-A_{23}A_{33}^{-1}z_1(t)-A_{23}A_{33}^{-1}B_3u(t)\big),
		\end{multline*}
		and we successively define $z_3\in L^2([0,T];\mathbb{R}^{r-r_1})$, $z_4\in L^2([0,T];\mathbb{R}^{r_1})$ by
		\begin{align*}
			z_3(t)=&\,-A_{33}^{-1}A_{23}z_2(t)-A_{33}^{-1}A_{13}z_1(t),\\
			z_4(t)=&\,-A_{14}^{-1}\big(\ddt E_{11}z_1(t)+\ddt E_{12}z_2(t)-A_{11}z_1(t)-A_{12}z_1(t)-A_{13}z_3(t)\big).
		\end{align*}
		Altogether, we have that $z:=(z_1^\tr,z_2^\tr,z_3^\tr,z_4^\tr)^\tr$ is a~solution of the differential-algebraic equation
		\begin{equation}
			\begin{aligned}
				\ddt    \begin{bmatrix}
					E_{11}&E_{12}&0&0\\E_{12}^\tr&E_{22}&0&0\\0&0&0&0\\0&0&0&0
				\end{bmatrix}
				\begin{pmatrix}
					z_{1}(t)\\z_{2}(t)\\z_3(t)\\z_4(t)
				\end{pmatrix}=&
				\begin{bmatrix}
					A_{11}&A_{12}&A_{13}&A_{14}\\A_{21}&A_{22}&A_{23}&0\\A_{31}&A_{32}&A_{33}&0\\-A_{14}^\tr&0&0&0
				\end{bmatrix}
				\begin{pmatrix}
					z_{1}(t)\\z_{2}(t)\\z_3(t)\\z_4(t)
				\end{pmatrix}+ \begin{bmatrix}
					B_{1}\\B_{2}\\B_{3}\\B_{4}
				\end{bmatrix}u(t),\\
				\begin{bmatrix}
					E_{11}&E_{12}&0&0\\E_{12}^\tr&E_{22}&0&0\\0&0&0&0\\0&0&0&0
				\end{bmatrix}
				\begin{pmatrix}
					z_{1}(0)\\z_{2}(0)\\z_3(0)\\z_4(0)
				\end{pmatrix}=&\begin{bmatrix}
					E_{11}&E_{12}&0&0\\E_{12}^\tr&E_{22}&0&0\\0&0&0&0\\0&0&0&0
				\end{bmatrix}
				\begin{pmatrix}
					z_{10}\\z_{20}\\z_{30}\\z_{40}
				\end{pmatrix}.\label{eq:daerel}
			\end{aligned}
		\end{equation}
		Now setting
		\[x(t):=(Q^\tr)^{-1}(V^\tr)^{-1}z(t),\]
		we obtain from \eqref{eq:pencrel} and \eqref{eq:daerel} that $x$ is a~solution of \eqref{eqn:DAEsystemQ}.
				\item The result has been shown in \cite[Thm.~4.3]{Mehl_2018} and \cite[Thm.~6.6]{GeHaRe20}. 
				 Notice, this can be also seen from the proof of (b). 
		\qedhere
	\end{enumerate}
\end{proof}
\begin{remark}
	Assume that the DAE \eqref{eqn:DAEsystemQ} fulfills Assumption~\ref{ass:overall-ivp}\,\eqref{ass:overall-ivpa} \& \eqref{ass:overall-ivpb}. By using
	$V\in\mathbb{R}^{n\times n}$ as defined as in \eqref{eq:Tdef}, we can transform \eqref{eqn:DAEsystemQ} to an equivalent DAE \eqref{eq:daerel}. Consequently, the existence of a~solution of \eqref{eqn:DAEsystemQ} implies $z_1\in H^1([0,T];\mathbb{R}^{r_1})$ with, and thus
	$B_4u=(A_{14}^\tr)^{-1}\in H^1([0,T];\mathbb{R}^{r_1})$. Invoking that $B_4=Z_1^\tr Z^\tr Q^\tr B$, we obtain that
	$Z_1^\tr Z^\tr Q^\tr Bu\in H^1([0,T];\mathbb{R}^{r_1})$ with $Z_1^\tr Z^\tr Q^\tr (JQx_0+Bu(0))=0$ has to be fulfilled. Consequently, under Assumption~\ref{ass:overall-ivp} \eqref{ass:overall-ivpa} \& \eqref{ass:overall-ivpb}, the conditions in Assumption~\ref{ass:overall-ivp}\,\eqref{ass:overall-ivpc} are also necessary for the existence of a~solution of the DAE \eqref{eqn:DAEsystemQ}.
	\hfill $\Box$
\end{remark}

\subsection{Perspective as coupled PH-subsystems}
\label{sec:coupled-ph-subsystems}
We take the perspective of the paper~\cite{Guenther2021} for the system \eqref{eqn:DAEsystemQ} and its partitioning given in Rem.~\ref{remark:partitioned-PHS}. To this end, we need to introduce the internal inputs $u_{ij}$, i.e., data stemming from the $j$th system being input for the $i$th system, and outputs $y_{ij}$ (i.e., data from the $i$ system to be transferred to the $j$th system).
This read for the $i$th subsystems ($i=1,\ldots,s$):
\begin{align*}
	\ddt E_i \varsplit_i(t) &= (J_i - R_i ) Q_i \varsplit_i(t)  
	+  \sum_{\begin{smallmatrix}j=1\\ j\neq i\end{smallmatrix}}^{s} B_{ij} {u}_{ij} (t)
	+ B_i u_i (t), \qquad E_ix_i(0)=Ex_{i0},
	\\
	y_{ij} (t) & =  B_{ij}^\tr Q_i \varsplit_i(t),\qquad j=1,\ldots,s,\, j\neq i, \\
	y_i (t) &= B_i^\tr Q_i \varsplit_i(t)
	\nonumber
\end{align*}
with matrices $E_i,Q_i,J_i,R_i \in \mathbb{R}^{n_i\times n_i}$, and $B_i\in \mathbb{R}^{n_i\times m}$ fulfilling the properties of  Assumption~\ref{ass:overall-ivp}\,\eqref{ass:overall-ivpa}, and $B_{ij}\in \mathbb{R}^{n_i\times m_{ij}}$ for $i=1,\ldots,s$, $j=1,\ldots,s$, $i\neq j$ with $m_{ij}=m_{ji}$.
The dissipation inequality for this subsystem reads
\begin{multline*}
	\tfrac12\big(x_i(t)^\tr Q_i^\tr E_ix_i(t)\big)-\tfrac12\big(x_i(0)^\tr Q_i^\tr E_ix_i(0)\big)\\\leq \int_0^tu_i(\tau)^\tr y_i(\tau)d\tau+\sum_{\begin{smallmatrix}j=1\\j\neq i\end{smallmatrix}}^s\int_0^tu_{ij}(\tau)^\tr y_{ij}(\tau)d\tau\qquad \forall\, t\in[0,T]
\end{multline*}
The overall input is $u=u_1=\ldots=u_s$, whereas the overall output is given by the sum
\[y=\sum_{j=1}^sy_i.\]
Coupling is done such that $x\mapsto x^\tr Q^\tr Ex$ with $E$ and $Q$ as in \eqref{eq:EQBprop} is a~storage function for the overall system \cite{CervScha07}. This is for instance achieved by 
\begin{equation*}
	\begin{aligned}
		\forall\,i=1,\ldots,s:\;&   u_{ij} + y_{ji} =0,\qquad j=1,\ldots,i-1,\\
		&   u_{ji} - y_{ij} =0,\qquad j=i+1,\ldots,s,
	\end{aligned}
\end{equation*}
since
\begin{align*}
	&\,\tfrac12\big(x(t)^\tr Q^\tr Ex(t)\big)-\tfrac12\big(x(0)^\tr Q^\tr Ex(0)\big)\\
	=&\,\sum_{i=1}^s\tfrac12\Big(\big(x_i(t)^\tr Q_i^\tr E_ix_i(t)\big)-\tfrac12\big(x_i(0)^\tr Q_i^\tr E_ix_i(0)\big)\Big)\\
	\leq& 
	 \sum_{i=1}^s\int_0^tu_i(\tau)^\tr y_i(\tau)d\tau+\sum_{i=1}^s\sum_{\begin{smallmatrix}j=1\\j\neq i\end{smallmatrix}}^s\int_0^tu_{ij}(\tau)^\tr y_{ij}(\tau)d\tau
	\\
	=& \int_0^t \!\! u(\tau)^{\!\tr} \! \sum_{i=1}^s y_i(\tau)d\tau
	  \!+\! \sum_{i=1}^s \sum_{j=1}^{i-1}\! \int_0^t \!\! u_{ij}(\tau)^{\!\tr}\! y_{ij}(\tau)d\tau
    	\!+\!\sum_{i=1}^s \sum_{j=i+1}^{s} \! \int_0^t \!\! u_{ij}(\tau)^{\!\tr} y_{ij}(\tau)d\tau
	\\
	=& \int_0^t \!\! u(\tau)^{\!\tr}\!  y(\tau) d\tau 
	+ \sum_{i=1}^s\sum_{j=1}^{i-1} \!\int_0^t \!\!\! -y_{ji}(\tau)^{\!\tr} u_{ji}(\tau)d\tau
	+\sum_{i=1}^s\sum_{j=i+1}^{s} \! \int_0^t \!\! u_{ij}(\tau)^{\!\tr}\! y_{ij}(\tau)d\tau
	\\
	=&\, \int_0^tu(\tau)^\tr y(\tau)d\tau.\end{align*}
The overall system is given by \eqref{eqn:DAEsystemQ} with $E$, $Q$ and $R$ as in \eqref{eq:EQBprop}, and $J$ structured as in 
\eqref{eq:EQBprop2} with
\begin{align*}
	\forall\, i=1,\ldots,s-1,\;j=i+1,\ldots, s:&\quad J_{ij}=B_{ij} B_{ji}^\tr.
\end{align*}

\begin{remark}
	This coupling has the drawback that all variables enter the definition of the coupling. And we have as many outputs of a subsystem as the other subsystem has variables times the number of other subsystems. This will be treated differently in practice. A 'sparse' coupling has to be introduced.
\end{remark}

Before we investigate the Lions/Mercier based algorithm, we revisit the Jacobi dynamic iteration and discuss some limitations.

\section{Failure of Jacobi dynamic iteration}
\label{chapter.failure}
To motivate the development of monotone operator splitting based dynamic iteration schemes in chapter~\ref{sec:dynamic-iteration-scheme}, we remind that classical approaches such as the Jacobi or Gauss-Seidel iteration scheme may not always be feasible for a computer implementation. 
To this end, we address system~\eqref{eqn:DAEsystemQ} in the following ODE setting: assigning $E=Q=I$, the  initial value problem~\eqref{eqn:DAEsystemQ} reads
\begin{align}
\label{eq:ivp-numerics}
    \dot x & = (J-R) x + B u(t), \quad x(0)=x_0, \quad t \in [0,T].
\end{align}
For this system, we recall that the dynamic iteration is based on a  splitting $J-R=M-N$ and an iteration count $k\in \nats_0$. Now, let be an initial guess $k=0$ be given $x^{(0)} \in C^{0}([0,T], \real^n)$ with $x^{(0)}(0)=x_0$ (e.g., $x^{(0)}(t)=x_0$) for all $t\in[0,T]$, the dynamic iteration scheme reads:
\begin{align}
\label{jac.it.ode}
    \dot x^{(k+1)} & = M x^{(k+1)} - N x^{(k)} + B u(t), \quad x^{(k+1)}(0)=x_0.
\end{align}
\begin{remark}
Note that setting $M=J_d-R$ and $N=-J_o$ in the case of the pHS-ODE system~\eqref{eq:ivp-numerics} defines a block-Jacobi dynamic iteration scheme.
\hfill $\Box$
\end{remark}

If we restrict to $u   \in C^{0}([0,T],\real^m)$, we obtain from \eqref{jac.it.ode} $x^{(k+1)} \in C^{0}([0,T],\real^n)$. Let $x \in C^{0}([0,T],\real^n)$ denote the analytic solution of \eqref{eq:ivp-numerics}.
Then, the recursion error is given by $\epsilon^{(k)}:=x^{(k)}-x$. 
Following Burrage~\cite{Burrage1995}, we can derive an exact error recursion for the dynamic iteration~\eqref{jac.it.ode}. Doing this, for the system~\eqref{eq:ivp-numerics} with 
\begin{equation}\label{eq:restriction-for-error-recursion}
  R=\diag(\tau,\dotsc,\tau), \quad B=0, 
\end{equation}
and setting $M=-R, N=-J_o=-J$, we get the error recursion on
$t \in [0,\, T]$: 
\begin{align*}
&\epsilon^{(k)}\!(t) =  
    \int_0^t \!\! e^{-(t-s_{k\!-\!1}) R } J \epsilon^{(k-1)}(s_{k\!-\!1}) 
  \diff s_{k-1} 
    = J \!\!\int_0^t \!\! e^{-\tau (t-s_{k\!-\!1}) }  \epsilon^{(k-1)}(s_{k-1}) \diff s_{k-1}
\\[0.5ex]
	&  = %
        J^{k} \!\!		
		\int_{0}^{t} \!\! e^{-\tau (t-s_{k-1}) } \!\!
		\int_{0}^{s_{k\!-\!1}}  \!\!\! e^{-\tau (s_{k\!-\!1}\!-\!s_{k\!-\!2}) } \!\!
		    \dotsc  \!
		    \int_0^{s_1} \!\!\!e^{-\tau (s_1\!-\!s_0) } \epsilon^{(0)} (s_0) \diff s_0  \dotsc \diff s_{k-2} \diff s_{k-1} 
		\\[0.5ex]
	&  =  e^{-\tau t} %
						J^{k} 
			\int_{0}^{t}  \int_{0}^{s_{k-1}}  \dotsi  \int_0^{s_1} e^{\tau  s_0 } \epsilon^{(0)} (s_0) \diff s_0  \dotsi \diff s_{k-2} \diff s_{k-1} .
\end{align*}
Now, just for demonstration purposes and simplicity, let us assume that initial error has the format: 
\[ \epsilon^{(0)} (s_0) = D   s_0 e^{-\tau s_0} \ones \; \text{ with } \ones=[1,\, 1]^{\tr}\]
and $D>0$. 
Using the maximum norm $\|f\|_T:=\max_{t \in [0,T]} \|f(t)\|_2$ for $f\in C^{0}([0,T],\real^n)$, the error reads for all $k\ge \tau T -1$ (maximal error at $T$): 
	\[
		\|\epsilon^{(k)}\|_T = \sqrt{2} \, D \|J\|_2^k    
		                       \frac{e^{-\tau T} T^{k+1}}{(k+1)!}  
		\quad
		\Rightarrow
		\quad
		\frac{\|\epsilon^{(k+1)}\|_T }{\|\epsilon^{(k)}\|_T}
		 = \frac{\|J\|_2 \cdot T}{k+2}   
		 .
	\]

\begin{example}
We analyse a $2\times 2$ version of \eqref{eq:ivp-numerics} with the restriction \eqref{eq:restriction-for-error-recursion}, which reads:
\begin{align}\label{eq:system-simple}
 x'=(J-R)x, \quad x(0)=x_0=\begin{bmatrix}
     2 \\2
 \end{bmatrix}, \qquad  \text{with } 
J=\begin{bmatrix}
            0 & \nu \\
            -\nu & 0
   \end{bmatrix}\!, \;
     R = \begin{bmatrix}
        \tau & 0 \\
            0 & \tau
   \end{bmatrix}
   \end{align}
and employing $\tau, \nu >0$. The analytic solution
and error evolution are given by:
\[
  x(t) = e^{-\tau t} \begin{bmatrix}
   \cos(\nu t) & \sin(\nu t) \\
   -\sin(\nu t) & \cos(\nu t)
  \end{bmatrix} x_0 
 , \qquad 
		\frac{\|\epsilon^{(k+1)}\|_T }{\|\epsilon^{(k)}\|_T}
	  = \frac{\nu \, T}{k+2}   \quad \forall k\ge \tau T -1 .
\]
	The error amplification factor $\nu T/(k+2)$ is bounded by one for all $k>\nu T -2$. For all $k < \nu T -2$ the error is increasing. 
	Hence, the convergence is monotone (for all $k$) only if $\nu T -2   \le 0$. 
	We give results for both cases:
\begin{enumerate}[(a)]
	    \item For $\nu T >2$, we obtain a decrease of the error only if $k$ is large enough. On the other hand, we may %
	    trigger overflows if $\nu T$ is large enough, see 
	    Fig.~\ref{fig:error-jacobi}, left.

	   \item If $\nu T \le 2$, we have monotone convergence. Hence windowing is always a solution to obtain (good) convergence. 

        An example for $T=0.5$ is shown in Fig.~\ref{fig:error-jacobi}, right. Thus using Jacobi with windows, we can obtain convergences. Note that the behavior in the  window $[0, 0.5]$ yields a kind of worst case for the convergence. 
        \hfill $\Box$
	\end{enumerate}
\end{example}

\begin{figure}
    \centering
    \includegraphics[width=0.485\textwidth]{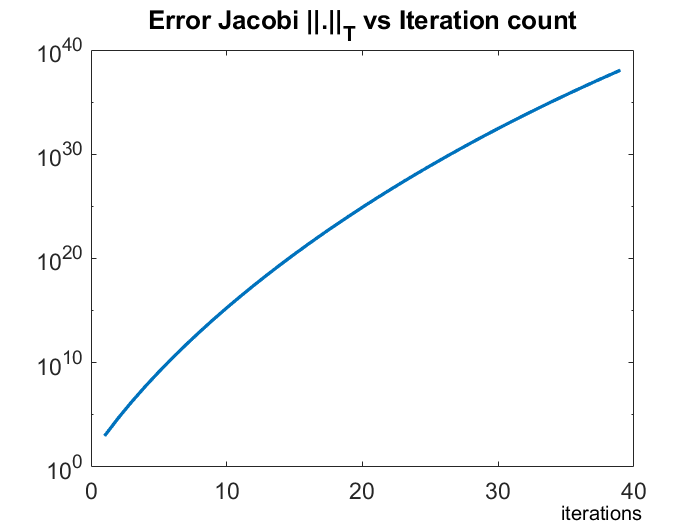}
    \hspace*{1ex}
    \includegraphics[width=0.485\textwidth]{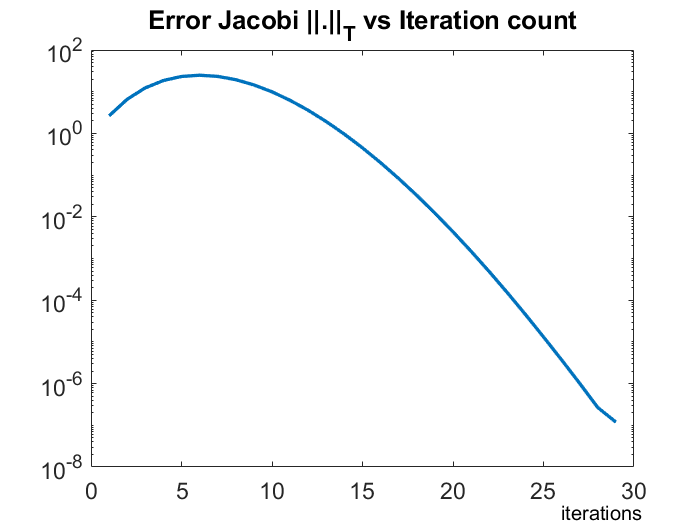}
        
    \caption{Jacobi iteration for system~\eqref{eq:system-simple} with $\nu=15$, $\tau=0.01$, $D=1$. Left: $T=10$. The maximum error would be obtained at iteration $k=148$:  $\sqrt{2} \cdot  150^{148}/148!$ and error decrease from $k=149$ on. Overflow in single precision at $k=40$.\newline
    Right: $T=0.5$. Maximum error for $k=5$ ($\nu T -2=5.5$), then decrease. After $k=29$, we achieve the solution (in single precision).}
    \label{fig:error-jacobi}
\end{figure}

\section{Recap on maximal monotone operators}
For $\omega\ge0$, we introduce the weighted $L^2$-space
$$ 
    L^2_\omega([0,T],\mathbb R^n):= 
        \{f:[0,T]\rightarrow \mathbb R^n \mid t\mapsto e^{-t \omega} f(t) \in  L^2([0,T],\mathbb R^n)\}
$$
equipped with the norm $\| \cdot \|_{2,\omega}$ defined as
$$ 
   \|f\|_{2,\omega}^2= 
  \int_0^T e^{-2t \omega} \| f(t)\|^2dt
  .
$$
Clearly, as  $T\in (0,\infty)$, we have $L^2_\omega([0,T],\mathbb R^n)=L^2([0,T],\mathbb R^n)$ with equivalent norms. Furthermore, if $\omega=0$, then the norms in $L^2_\omega([0,T],\mathbb R^n)$ and $L^2([0,T],\mathbb R^n)$ coincide, i.e., 
\[\forall\,f\in  L^2([0,T],\mathbb R^n):\quad \|f\|_{2,0}=\|f\|_{2}= \left(\int_0^T  \| f(t)\|^2dt\right)^{1/2}.\]

\begin{definition}[Contraction]
Let $X$ be a Banach space space with norm $\| \cdot\|$. A (possibly nonlinear) operator $A:D(A)\subset X\to X$ is called {\em contractive}, if
\[
   \|Ax-Ay\|\leq \|x-y\|, \qquad \forall\, x,y\in X,
\]
and {\em strictly contractive}, if there exists some $L<1$, such that
\begin{equation*}
   \|Ax-Ay\|\leq L\,\|x-y\|, \qquad \forall\, x,y\in X .
   \tag*{$\Box$}
\end{equation*}
\end{definition}

\begin{definition}[Maximally monotone operator]
Let $X$ be a real Hilbert space with inner product $\langle \cdot, \cdot\rangle$. A set $Y\subset X\times X$ is called {\em monotone}, if
\[
   \langle x-u,y-v \rangle\ge 0, \qquad (x,y),(u,v)\in Y.
\]
Further, $Y\subset X\times X$ is called  {\em maximally monotone}, if it is monotone and not a proper subset of a monotone subset of $X\times X$.\\
A (possibly nonlinear) operator $A:D(A)\subset X\rightarrow X$ is called {\em (maximally) monotone}, if the graph of $A$, i.e., $\{(x,Ax):\, x\in D(A)\}$, is (maximally) monotone. 
\hfill $\Box$
\end{definition}

\begin{remark}\label{rem:inv}
Let $A:D(A)\subset X\rightarrow X$ be a~monotone operator. 
It follows from the definition of monotonicity that $I+\lambda A$ is injective for all $\lambda>0$.\\
Moreover, by \cite[Theorem 2.2]{Bar10}, the following three statements are equivalent: 
\begin{enumerate}[(i)]
\item $A$ is maximally monotone,
\item $I+\lambda A$ is surjective for some $\lambda>0$,
\item $I+\lambda A$ is surjective for all $\lambda>0$.
\end{enumerate}
Consequently, if $A$ is maximally monotone, then $I+\lambda A$ is bijective for all $\lambda>0$. The Cauchy-Schwarz inequality yields that
\[\|( I+\lambda A)x-( I+\lambda A)y\|\geq \|x-y\|\quad \forall x,y\in D(A),\]
whence $( I+\lambda A)^{-1}:X\rightarrow X$ is contractive. %

Furthermore, $(I-\lambda A)(I+\lambda A)^{-1}$ is contractive. 
 This follows with $\tilde x:= (I+\lambda A)^{-1}x$ and $\tilde y:= (I+\lambda A)^{-1}y$ from 
\begin{align*}
\|x-y\|^2-\|(I-\lambda A)&(I+\lambda A)^{-1}x -(I-\lambda A)(I+\lambda A)^{-1}y\|^2\\
& =\|\tilde x-\tilde y+\lambda A\tilde x-\lambda A\tilde y\|^2-\|\tilde x-\tilde y-(\lambda A\tilde x-\lambda A \tilde y)\|^2\\
& =4\lambda\langle \tilde x-\tilde y,A\tilde x-A\tilde y\rangle\ge0.
\tag*{$\Box$}
\end{align*}
\end{remark}

\section{Dynamic iteration scheme}
\label{sec:dynamic-iteration-scheme}

Now we develop a dynamic iteration scheme for the DAE \eqref{eqn:DAEsystemQ} fulfilling Assumption~\ref{ass:overall-ivp}.
To facilitate the decoupling, we split %
$J$ as follows
\begin{equation}
    J =\,\underbrace{\begin{bmatrix}
J_1&0&\cdots&0\\
0&\ddots&\ddots&\vdots\\
\vdots&\ddots&\ddots&0\\
0&\cdots&0 &J_s\end{bmatrix}}_{=:J_d}+\,\underbrace{\begin{bmatrix}
0&J_{12}&\cdots&J_{1s}\\
-J_{12}^\tr&\ddots&\ddots&\vdots\\
\vdots&\ddots&\ddots&J_{s-1,s}\\
-J_{1s}^\tr&\cdots&-J_{s-1,s}^\tr &0\end{bmatrix}}_{=:J_o}. \label{eq:Jsplit}
\end{equation}
such that $J_d = - J_d^\tr$ and $J_o=-J_o^\tr$ (compare with Rem.~\ref{remark:partitioned-PHS} \ref{remark:partitioned-PHS-subsystems}).

Let $\alpha\in [0,1]$, $\mu,\omega\in\mathbb{R}$ with $0\leq\mu\leq \omega$. We introduce 
the affine-linear operator $M$ by:
\begin{subequations}\label{eq:Mop}
\begin{alignat}{2}
M\!:&&\,D(M)\subset L^2_\omega([0,T],\mathbb R^n)&\rightarrow L^2_\omega([0,T],\mathbb R^n),\\
&&x &\mapsto\ddt (E{x}) -(J_d-\alpha R 
    +\mu EQ^{-1})Q x \!-\!B u
\end{alignat}
{with domain}
\begin{equation} D(M) = \left\{ x\in L^2_\omega([0,T],\mathbb R^{n})\mid
 Ex\in H^1([0,T],\mathbb R^{n}), E
 x(0)=Ex_{0}\right\},
 \end{equation}
\end{subequations}
 and the linear bounded multiplication operator $N$ by
 \begin{subequations}\label{eq:Nop}
\begin{align}
N\!:&& L^2_\omega([0,T],\mathbb R^n)&\,\rightarrow L^2_\omega([0,T],\mathbb R^n),\\
&&    x&\,\mapsto
        ((1-\alpha)R+\mu EQ^{-1}-J_o) Qx.
 \end{align}
\end{subequations}

\begin{remark}\label{rem:splitting}
Notice that $x$ is a solution of the DAE~\eqref{eqn:DAEsystemQ}  if, and only if,
\begin{equation}\label{equationQbalg}
   Mx + Nx = 0
\end{equation}
holds. \mbox{}\hfill $\Box$
\end{remark}

  \begin{lemma}\label{lem:QM0}
 Let $\omega,\mu\in\mathbb{R}$ with $0\leq\mu\leq\omega$, and let $\alpha\in[0,1]$. Assume that the DAE \eqref{eqn:DAEsystemQ} fulfills Assumption~\ref{ass:overall-ivp}, and let the operator $M$ be defined as in \eqref{eq:Mop}.  Then $MQ^{-1}$ is maximally monotone %
 on $L^2_\omega([0,T],\mathbb R^n)$. More precisely, $MQ^{-1}$ satisfies
 \begin{align*}
     \langle MQ^{-1} v & - MQ^{-1} w,\;  v-w \rangle_{2,\omega} 
     \\
   &  =
      \tfrac{1}{2} e^{-2 T\omega }\left(   
                                  \|(EQ^{-1})^{1/2}(v(T)-w(T))\|^2 
                                \right) 
      \\
      &\qquad (\omega-\mu) \|(EQ^{-1})^{1/2}(v-w)\|^2_{2,\omega} +\alpha\| R^{1/2}(v-w)\|^2_{2,\omega}\ge0
 \end{align*} 
 for $v, w\in Q D(M)$, and $I+\lambda MQ^{-1}$ is surjective for every $\lambda>0$. 
 \end{lemma}
 \begin{proof}
  First of all, we have  for $v, w\in Q D(M)$
 \begin{align*}
  &\langle MQ^{-1} v - MQ^{-1} w,v-w\rangle_{2,\omega}
  \\
  &=  \langle \ddt EQ^{-1}(v-w),\, v-w\rangle_{2,\omega} 
     -\langle (J_d-\alpha R +\mu EQ^{-1})(v-w), v-w\rangle_{2,\omega}
     \\
  & = \langle \ddt EQ^{-1}(v\!-\!w), v-w\rangle_{2,\omega}
     - \mu \langle (EQ^{-1}(v\!-\!w),
              v\!-\!w\rangle_{2,\omega}
     + \langle \alpha R(v\!-\!w), v\!-\! w\rangle_{2,\omega}
     \\
   &=\tfrac{1}{2} e^{-2 T\omega }\left(\|(EQ^{-1})^{1/2}(v(T)-w(T))\|^2 \right) 
     \\
   &\qquad+ (\omega-\mu) \|(EQ^{-1})^{1/2}(v-w)\|^2_{2,\omega}+\langle \alpha R(v-w), v-w\rangle_{2,\omega}\ge0,
  \end{align*}
which shows that $MQ^{-1}$ is monotone.\\
To prove that $MQ^{-1}$ is maximally monotone, it suffices, by Remark~\ref{rem:inv}, to prove that $I+\lambda MQ^{-1}$ is surjective for some $\lambda>0$. Assume that $y\in L^2_\omega([0,T],\mathbb R^n)$. Consider the matrices $\tilde{E}=EQ^{-1}$, $\tilde{R}=\alpha R
          +\frac1\lambda I$ and $\tilde{Q}=I$. Then the DAE
\begin{equation} \ddt \tilde{E}\tilde{v}(t) =
(J_d-\tilde{R})\tilde{Q}\tilde{v}(t) +B e^{-\mu t}u(t)-\tfrac1\lambda e^{-\mu t} y(t),\quad \tilde{E}\tilde{v}(0)=\tilde{E}Qx_{0},\label{eq:daeshift}
\end{equation}
This DAE clearly fulfills Assumption~\ref{ass:overall-ivp}\,\eqref{ass:overall-ivpa}\&\eqref{ass:overall-ivpb}. Moreover, by positive definiteness of $\tilde{R}$, we obtain that any matrix $Z\in\mathbb{R}^{n\times r}$ with full column rank and $\IM Z=\ker E$ fulfills $\ker \tilde{R}QZ=\{0\}$.
This means that Assumption~\ref{ass:overall-ivp}\,\eqref{ass:overall-ivpc}
is trivially fulfilled by \eqref{eq:daeshift}, and 
Proposition~\ref{prop:DAEbasicprop} implies that \eqref{eq:daeshift} has a~unique solution $\tilde{v}\in L^2_\omega([0,T],\mathbb R^n)$ with $\tilde{E}\tilde{v}\in H^1([0,T],\mathbb R^n)$. Now consider  ${v}\in L^2_\omega([0,T],\mathbb R^n)$ with $v(t)=e^{\mu t}\tilde{v}(t)$. Then $\tilde{E}v(0)=\tilde{E}\tilde{v}(0)=\tilde{E} Qx_0$, and
the product rule for weak derivatives \cite[Thm.~4.25]{Alt16} yields $\tilde{E}{v}\in H^1([0,T],\mathbb R^n)$ with
\[\begin{aligned}
\ddt \tilde{E}v=&\,\ddt e^{\mu t}\tilde{E}\tilde{v}\\
=&\,
 e^{\mu t}\ddt \tilde{E}\tilde{v}+\mu e^{\mu t} \tilde{E}\tilde{v}\\
 =&\,
e^{\mu t}(J_d-\tilde{R})\tilde{Q}\tilde{v} +e^{\mu t}B e^{-\mu t}u(t)-e^{\mu t}\tfrac1\lambda e^{-\mu t} y(t)+\mu e^{\mu t} \tilde{E}\tilde{v}\\
=&\, (J_d-\tilde{R})\tilde{Q}{v} +B u(t)-\tfrac1\lambda y(t)+\mu  \tilde{E}{v},    
\end{aligned}
\]
and thus
\[v+\lambda\big(\ddt {E}Q^{-1}v-(J_d-\alpha R+\mu EQ^{-1})QQ^{-1}v+Bu)=y.\]
Since this means that $(I+\lambda MQ^{-1})v=y$, the result is shown.
 \end{proof}

\begin{lemma}\label{lem:cayleyalg}
Let $\lambda>0$, $\mu\geq0$, $\alpha\in[0,1]$, and let $R,E,Q\in\mathbb{R}^{n\times n}$ with the properties as in  Assumption~\ref{ass:overall-ivp}. Then, for $K:=((1-\alpha)R+\mu EQ^{-1}-J_o)$
\begin{multline}\label{eq:estcayl1}
    \|(I-\lambda K)(I+\lambda K)^{-1}x\|^2-\|x\|^2\\ =-4{\lambda} \|((1-\alpha)R+\mu EQ^{-1})^{1/2}(I+\lambda K)^{-1}x\|^2, \quad\forall\, x\in \mathbb R^n .
\end{multline}
Moreover, if 
\begin{equation}\mathrm{rk}\begin{bmatrix}
        \mu E &
        (1-\alpha)R 
      \end{bmatrix}=n,\label{eq:rkdisscond}
      \end{equation}
then $(1-\alpha)R+\mu EQ^{-1} = K+J_o\in\mathbb{R}^{n\times n}$ is invertible, and 
\begin{equation}\|(I-\lambda K)(I+\lambda K)^{-1}x\|\leq q \, \|x\|\quad\forall\, x\in \mathbb R^n,\label{eq:estcayl2}\end{equation}
where $q\ge 0$ with
\begin{equation}
    q^2 =1-\frac{4\lambda}{(1+\lambda\|K\|)^2 \cdot {\|(K+J_o)^{-1}\|}}<1.
     \label{eq:estcayl3}
\end{equation}
\end{lemma}
\begin{proof}
By using that $K+K^\tr\ge0$, we see that $I+\lambda  K$ is an invertible matrix.
Let $x\in\mathbb{R}^n$ and $y:=(I+\lambda  K)^{-1}x$. Then
\begin{multline*}
\|x\|^2-\|(I-\lambda  K)(I+\lambda  K)^{-1}x\|^2\\
=\|y+\lambda  K y\|^2-\|y-\lambda  K y\|^2
=4\lambda\langle y,  (K+K^\tr) y\rangle\\
=4\lambda\langle y, \big( (1-\alpha)R+\mu EQ^{-1}\big) y\rangle
\end{multline*}
shows \eqref{eq:estcayl1}. To prove the remaining result, assume that \eqref{eq:rkdisscond} is fulfilled. Then it can be concluded from Assumption~\ref{ass:overall-ivp}\,\eqref{ass:overall-ivpa}\,(i) that $(1-\alpha)R+\mu EQ^{-1}$ is positive definite.
Furthermore, by
\[
     \| x\| 
    =
    \|(I+\lambda K) y\| \leq(1+\lambda \|K\|) \|y\|
     = (1+\lambda \|K\|) \| (I+\lambda  K)^{-1}x \|,
\]
we have
\[
   \|(I+\lambda K)^{-1}x\| \geq\frac{1}{1+ \lambda\|K\|} \,\|x\|\quad\forall\, x\in\mathbb{R}^n.
\]
Plugging this into \eqref{eq:estcayl1}, we obtain for all $x\in \mathbb R^n$
\begin{align*}
    \|(I-\lambda K)(I+\lambda K)^{-1}x\|^2\!\!\!\!\!\!\!\!\!\!\!\!\!\!\!\!\!\!\!\!\!\!\!&
    \\ 
      =&\,\|x\|^2-4{\lambda} \|((1-\alpha)R
                +\mu EQ^{-1})^{1/2}(I+\lambda K)^{-1}x\|^2\\
    \leq&\, \|x\|^2-4\lambda\|((1-\alpha)R+\mu EQ^{-1})^{-1}\|^{-1} \|(I+\lambda K)^{-1}x\|^2\\
    \leq&\, \|x\|^2-4\lambda\|((1-\alpha)R+\mu EQ^{-1})^{-1}\|^{-1} \frac{1}{(1+ \lambda\|K\|)^2} \,\|x\|^2,\\
    =&\,\left( 1- \frac{4\lambda}{\|((1-\alpha)R+\mu EQ^{-1})^{-1}\| \cdot (1+\lambda\|K\|)^2}\right) \,\|x\|^2,
\end{align*}
which completes the proof.
\end{proof}
\begin{remark}
Condition \eqref{eq:rkdisscond} is equivalent to at least one of the following three statements being fulfilled:
\begin{enumerate}[(i)]
\item $\mathrm{rk}\begin{bmatrix}
        E &
        R 
      \end{bmatrix}=n$, $\mu>0$ and $\alpha <1$,
\item $\mathrm{rk}E=n$ and $\mu>0$, or
\item $\mathrm{rk}R=n$ and $\alpha <1$ \hfill $\Box$.
\end{enumerate}
\end{remark}

As a direct conclusion of Lemma~\ref{lem:cayleyalg}, we can draw the following result.

 \begin{lemma}\label{lem:QNalg}
Let $\lambda,\mu>0$, $\alpha\in[0,1]$, let $R,E,Q\in\mathbb{R}^{n\times n}$ with the properties as in  Assumption~\ref{ass:overall-ivp}, and let the operator $N$ be defined as in \eqref{eq:Nop}. The multiplication operator $NQ^{-1}$ on $L^2_\omega([0,T],\mathbb R^n)$ is linear, bounded and satisfies
\[ \langle x, N Q^{-1} x\rangle_{2,\omega} \leq  0,\]
i.e., $NQ^{-1}$ is maximally monotone.
Furthermore,  the multiplication operator $x\mapsto (I-\lambda NQ^{-1})(I+\lambda NQ^{-1})^{-1}x$ is a  contraction on $L^2_\omega([0,T],\mathbb R^n)$ for $\lambda >0$ satisfying
\[\begin{aligned}\|(I-\lambda NQ^{-1})&(I+\lambda NQ^{-1})^{-1}x\|_{2,\omega} \leq \|x\|_{2,\omega}\quad\forall \,x\in L^2_\omega([0,T],\mathbb R^n).
\end{aligned}\]
Moreover, if $\IM E+\IM R=\mathbb{R}^n$ and $\alpha <1$, then
\[\begin{aligned}\|(I-\lambda NQ^{-1})&(I+\lambda NQ^{-1})^{-1}x\|_{2,\omega} 
    \leq q \|x\|_{2,\omega}\quad\forall \,x\in L^2_\omega([0,T],\mathbb R^n).
\end{aligned}\]
for $q<1$ as in \eqref{eq:estcayl3}.
\end{lemma}
\begin{proof}
The first statement follows directly and the remaining statements are implied by Lemma~\ref{lem:cayleyalg}.
 \end{proof}

Our dynamic iteration scheme is built on Rem.~\ref{rem:splitting}.
Following Lions and Mercier~\cite{Lions1979}, we study the algorithms
\begin{align}\label{algoQxn}
x^{k+1}:=& Q^{-1}(I+\lambda MQ^{-1})^{-1}(I-\lambda NQ^{-1})(I+\lambda NQ^{-1})^{-1}(I-\lambda MQ^{-1})Q x^{k}.
\end{align}
with $x^0\in D(M)$ arbitrary. 
\begin{remark}\label{rem:solution}
Note that  equation \eqref{equationQbalg}  is equivalent to
\[(I+\lambda NQ^{-1})(I+\lambda MQ^{-1})Qx= (I-\lambda NQ^{-1})(I-\lambda MQ^{-1})Qx,\]
or equivalently,
\[x= Q^{-1}(I+\lambda MQ^{-1})^{-1}(I-\lambda NQ^{-1})(I+\lambda NQ^{-1})^{-1}(I-\lambda MQ^{-1})Qx.
\tag*{$\Box$}
\]
\end{remark}

\begin{remark}
The iteration \eqref{algoQxn} is equivalent to
\begin{align}
x^{k+1}:=& (Q+\lambda M)^{-1}(Q-\lambda N)(Q+\lambda N)^{-1}(Q-\lambda M)x^{k}. \label{alg:xcontr}
\end{align}
\end{remark}

\begin{theorem}\label{Theo:1}
Let $\lambda>0$, $\mu\geq0$, and let 
$\alpha\in[0,1]$. Assume that the DAE \eqref{eqn:DAEsystemQ} 
fulfills Assumption~\ref{ass:overall-ivp}, and let the operators
$M$, $N$ be defined as in \eqref{eq:Mop} and \eqref{eq:Nop}.
Let $\lambda>0$ and $x^0\in D(M)$. Let the sequence $(x^k)_k$ be defined by \eqref{algoQxn}, and let
$x$ be the unique solution  of \eqref{eqn:DAEsystemQ}.
Then the following statements are fulfilled:
\begin{enumerate}[(a)]
\item\label{Theo:1a} $(E x^k)_k$ converges 
to $Ex$ in $L^2([0,T],\mathbb R^n)$,
\item\label{Theo:1b} $(E x^k)_k$ converges pointwise to $Ex$ on $(0,T]$,
\item\label{Theo:1c} If $\alpha>0$, then $(RQ x^k)_k$ converges 
to $RQx$ in $L^2([0,T],\mathbb R^n)$,
\item\label{Theo:1d} For the sequence $(z^k)_k$ given by
$$\forall\,k\in\mathbb{N}:\quad z^{k}:= (I+\lambda MQ^{-1}) Qx^{k},$$
and the function $z:=(I+\lambda MQ^{-1})Q x$, the real sequence
 $(\|z^{k}-z\|_{2,\omega})_k$ is monotonically decreasing, and 
\begin{equation}
    \label{estimate.deltax}
\forall\,k\in\mathbb{N}:\quad  \|x^{k}- x\|_{2,\omega}  \le \|Q^{-1}\| \|z^{k}-z\|_{2,\omega}.
\end{equation}
\end{enumerate}
\end{theorem}
\begin{proof}
Let $\omega\in\mathbb{R}$ with $\mu<\omega$.
We have
\begin{align}
z^{k+1} &= (I-\lambda NQ^{-1})(I+\lambda NQ^{-1})^{-1}(2I-(I+\lambda MQ^{-1}))(I+\lambda MQ^{-1})^{-1}{z}^k\nonumber\\
&=(I-\lambda NQ^{-1})(I+\lambda NQ^{-1})^{-1}[2Q x^k-z^k]\label{iteration1alg}
\end{align}
and
$$ 
z= (I-\lambda NQ^{-1})(I+\lambda NQ^{-1})^{-1}(2Qx-z). 
$$
We set
\[
     \Delta x^{k}:= x-x^{k}, \qquad \Delta z^{k}:= z-z^{k}.\]
Then using Lemma \ref{lem:QNalg} and Lemma \ref{lem:QM0} we calculate 
\begin{align}
\|\Delta & z^{k+1}\|_{2,\omega}^2 
    -\|\Delta z^{k}\|_{2,\omega}^2 
    = \|z- z^{k+1}\|_{2,\omega}^2 -\|\Delta z^{k}\|_{2,\omega}^2 \nonumber
    \\
    &= \|(I-\lambda NQ^{-1})(I+\lambda NQ^{-1})^{-1}[(2
        Qx-z)- (2Qx^{k}-z^{k})]\|_{2,\omega}^2 -\|\Delta z^{k}\|_{2,\omega}^2
   \nonumber \\
    &\leq \|2 Q\Delta x^k-\Delta z^k\|_{2,\omega}^2 -\|\Delta z^{k}\|_{2,\omega}^2
    \nonumber \\
    & =4\|Q \Delta x^k\|_{2,\omega}^2 -4 \langle Q\Delta x^k,\Delta z^{k}\rangle_{2,\omega}
    \nonumber\\
    & =4\| Q\Delta x^k\|_{2,\omega}^2 -4 \langle Qx-Qx^k,(I+\lambda MQ^{-1})Qx- (I+\lambda MQ^{-1})Qx^{k}\rangle_{2,\omega}
    \nonumber\\
    & =-4\lambda  \langle Qx-Qx^k,Mx- Mx^{k}\rangle_{2,\omega}
    \nonumber\\
    &= -2\lambda e^{-2T\omega} \|(EQ^{-1})^{1/2} Q\Delta x^k(T)\|^2
    \nonumber\\
    &\quad- 4\lambda (\omega-\mu) \|(EQ^{-1})^{1/2} Q \Delta x^k\|_{2,\omega}^2- 4\lambda \alpha\|R^{1/2} Q \Delta x^k\|_{2,\omega}^2 .\label{eqn:ung}
\end{align}
Thus the sequence $(\|\Delta z^{k}\|_{2,\omega})_k$  monotonically decreasing, and therefore it  converges.
Inequality \eqref{estimate.deltax} follows from Remark \ref{rem:inv}, which completes the proof of \ref{Theo:1d}.
Furthermore, the latter inequality gives 
\begin{align}
    \|(EQ^{-1})^{1/2} Q \Delta x^k\|_{2,\omega}^2 \le&\, \frac{1}{4 \lambda (\omega-\lambda) }(\|\Delta z^{k}\|_{2,\omega}^2-\|\Delta z^{k+1}\|_{2,\omega}^2),\label{eq:firstineq}\\ 
\|R^{1/2} Q \Delta x^k\|_{2,\omega}^2 \le&\, \frac{1}{4 \lambda\alpha }(\|\Delta z^{k}\|_{2,\omega}^2-\|\Delta z^{k+1}\|_{2,\omega}^2).\label{eq:secondineq}
\end{align}
Combining \eqref{eq:secondineq} with the fact that the norms $\|\cdot\|_{\omega,2}$ and $\|\cdot\|_{2}$ are equivalent, we are led to \eqref{Theo:1c}. Further, \eqref{eq:firstineq} together with $E =(EQ^{-1})^{1/2} (EQ^{-1})^{1/2} Q$ gives
 $(\|E \Delta x^k\|_{2,\omega}^2 )_k$ converges to zero, and the equivalence between the norms $\|\cdot\|_{\omega,2}$ and $\|\cdot\|_{2}$ now implies \eqref{Theo:1a}.\\
 It remains to prove \eqref{Theo:1b}, using \eqref{eqn:ung} for $t\in(0,T]$ instead of $T$ we obtain
$$  \|(EQ^{-1})^{1/2} Q\Delta x^k(t)\|^2 \le \frac{e^{2t\omega}}{2\lambda}(\|\Delta z^{k}\|_{2,\omega}^2-\|\Delta z^{k+1}\|_{2,\omega}^2),$$
which shows that, for all $t\in(0,T]$,
$ Ex^k(t)\rightarrow E x(t).$
\end{proof}

\begin{remark}
Though the sequence 
$(\|\Delta x^k\|_{2,\omega})_k$ 
is usually not monotone decreasing, it is bounded by the  monotone decreasing sequence $(\|Q^{-1}\|\|z^{k}-z\|_{2,\omega})_k$ due to~\eqref{estimate.deltax}. \hfill $\Box$
\end{remark}

\begin{remark}
As the sequences $(x^k)_k$, $(z^k)_k$ are bounded in $L^2$, each of them has a~weakly convergent subsequence. %
\hfill $\Box$
\end{remark}

\begin{theorem}\label{thm2}
Let $\lambda>0$, $\mu\geq0$, and let $\alpha\in[0,1]$. Assume that the DAE \eqref{eqn:DAEsystemQ} fulfills Assumption~\ref{ass:overall-ivp}, and let the operators $M$, $N$ be defined as in \eqref{eq:Mop} and \eqref{eq:Nop}.
Let $\lambda>0$ and $x^0\in D(M)$. Let the sequence $(x^k)_k$ be defined by \eqref{algoQxn}, and let
$x$ be the unique solution  of \eqref{eqn:DAEsystemQ}. %
Then the following statements are fulfilled:
\begin{enumerate}[(a)]
\item\label{thm2:1b} If \eqref{eq:rkdisscond} holds,
then the sequence $(z^k)$ as defined in Theorem~\ref{Theo:1}~\eqref{Theo:1d} converges to $z:=(I+\lambda MQ^{-1})Q x$ in $L^2([0,T],\mathbb R^n)$. Further, $(Ex^k)_k$ converges 
to $Ex$ in $H^1([0,T],\mathbb R^n)$.
\item\label{thm2:1a} If \eqref{eq:rkdisscond} or 
\begin{equation}\mathrm{rk}\begin{bmatrix}
        E &
        \alpha R 
      \end{bmatrix}=n,\label{eq:rkdisscond2}
      \end{equation}
is fulfilled, then $(x^k)_k$ converges 
to $x$ in $L^2([0,T],\mathbb R^n)$,
\end{enumerate}
\end{theorem}
\begin{proof}
We start with the proof of \eqref{thm2:1b}: By \eqref{algoQxn} and Remark~\ref{rem:solution}, we have
\begin{align*}
z^{k+1}:=&\, (I-\lambda MQ^{-1})(I+\lambda MQ^{-1})^{-1}(I-\lambda NQ^{-1})(I+\lambda NQ^{-1})^{-1} z^{k},\\
z:=&\, (I-\lambda MQ^{-1})Q(I+\lambda MQ^{-1})^{-1}(I-\lambda NQ^{-1})(I+\lambda NQ^{-1})^{-1} z,
\end{align*}
and we define $\Delta z^{k}:= z-z^{k}$.
Assume that \eqref{eq:rkdisscond} holds. 
By invoking Remark~\ref{rem:inv}, Lemma~\ref{lem:QM0} and Lemma~\ref{lem:QNalg}, we obtain, for $q<1$ as in Lemma~\ref{lem:cayleyalg},
 \begin{align*}
\|\Delta z^{k+1}\|_{2,\omega} &= \|(I-\lambda MQ^{-1})(I+\lambda MQ^{-1})^{-1}(I-\lambda NQ^{-1})(I+\lambda NQ^{-1})^{-1} z\\&\quad-
(I-\lambda MQ^{-1})(I+\lambda MQ^{-1})^{-1}(I-\lambda NQ^{-1})(I+\lambda NQ^{-1})^{-1} z^{k}\|_{2,\omega}\\
&\leq \|(I-\lambda NQ^{-1})(I+\lambda NQ^{-1})^{-1} \Delta z^k\|_{2,\omega}\\
&\leq q\|\Delta z^k\|_{2,\omega}.
\end{align*}
This implies that $z^k$ converges in $L^2([0,T],\mathbb{R}^n)$ to $z$.\\
Next we prove that $(Ex^k)_k$ converges 
to $Ex$ in $H^1([0,T],\mathbb R^n)$. 
We already know from Theorem~\ref{Theo:1}\,\eqref{Theo:1a} that $(E x^k)_k$ converges 
to $Ex$ in $L^2([0,T],\mathbb R^n)$.
Hence, it suffices to prove that $(\ddt E x^k)_k$ converges 
to $\ddt Ex$ in $L^2([0,T],\mathbb R^n)$. By invoking that $(z^k)$ converges  to $z$ in $L^2([0,T],\mathbb R^n)$, \eqref{estimate.deltax} yields that $(x^k)$ converges to $x$ in $L^2([0,T],\mathbb R^n)$. Now using the definition of $x$, $z$ and $M$, we have
\[\begin{aligned}
    z^k=&\,Qx^k+\lambda \big(\ddt (E{x}^k) -(J_d-\alpha R+\mu EQ^{-1})Q x^k-B u),\\
    z=&\,Qx+\lambda \big(\ddt (E{x}) -(J_d-\alpha R+\mu EQ^{-1})Q x^k-B u),
\end{aligned}\]
which gives
\begin{multline*}
    \ddt (E{x})-\ddt (E{x}^k)=\tfrac1\lambda (z-z^k)+\tfrac1\lambda Q(x-x^k)+(J_d-\alpha R+\mu EQ^{-1})Q(x-x^k).
\end{multline*}
Since $(x^k)$ converges to $x$ in $L^2([0,T],\mathbb R^n)$, and further, $(z^k)$ converges to $z$ in $L^2([0,T],\mathbb R^n)$, the above equation implies that $\ddt (E{x}^k)$ converges in $L^2([0,T],\mathbb R^n)$ to $\ddt (E{x})$. Altogether, this means that $(Ex^k)$ converges to $Ex$ in $H^1([0,T],\mathbb R^n)$.

Next we prove \eqref{thm2:1a}. The case where \eqref{eq:rkdisscond} is fulfilled, the result follows by a~combination of \eqref{thm2:1b} with \eqref{estimate.deltax}. \\
Assume that \eqref{eq:rkdisscond2} holds: Then, by invoking Assumption~\ref{ass:overall-ivp}\,\eqref{ass:overall-ivpa}\,(i), the matrix
$Q^{\tr}E+\alpha Q^{\tr}RQ$ is positive definite, and thus invertible.
On the other hand,
$(((Q^{\tr}E+Q^{\tr}RQ) x^k)_k$ converges 
to $(Q^{\tr}E+\alpha Q^{\tr}RQ)x$ in $L^2([0,T],\mathbb R^n)$  by Theorem~\ref{Theo:1}. Consequently,  $(x^k)_k$ converges 
to $x$ in $L^2([0,T],\mathbb R^n)$.
\end{proof}

Based on the results of the previous section we develop the following iteration scheme. We recall that we choose $\omega>0$ and we define 
\[ x^{k+1}:= (Q+\lambda M)^{-1}(Q-\lambda N)(Q+\lambda N)^{-1}(Q-\lambda M)x^{k},\]
with $x^0 \in D(M)$ arbitrary, and 
\[ z^k:=(I+\lambda MQ^{-1})Qx^k.\]
Using \eqref{iteration1alg} we obtain 
\[ z^{k+1} = (I-\lambda NQ^{-1})(I+\lambda NQ^{-1})^{-1}[2Q x^k-z^k]. \]
Furthermore, relation $z^k=(Q+\lambda M)x^k$ corresponds to the DAEs
\[\tfrac{d}{dt}E_ix_i^k=(J_i-R_i  + \mu E_iQ_i^{-1} )Q_ix_i^k
   -\tfrac{1}{\lambda}Q_ix_i^k+  B_iu+\tfrac{1}{\lambda}z_i^k,  \quad i=1,\ldots,s,\]
with initial conditions $E_ix_i^n(0)=E_i x_{0,i}$ ($i=1,\ldots,s$).

\begin{remark}\
\begin{enumerate}[(a)]
\item If $\mathrm{rk}\begin{bmatrix}
        E &
        R 
      \end{bmatrix}<n$,
      then it is not possible to formulate any convergence results for $(x^k)$, in general. As an example, consider
      a~system \eqref{eqn:DAEsystemQ} with 
      \[E=R=0_{2\times 2},\;Q=I_2,\;J=\left[\begin{smallmatrix}0&1\\-1&0\end{smallmatrix}\right],\;u=0\in L^2([0,T];\mathbb{R}^2),\]
      which belongs to the class specified in Assumption~\ref{ass:overall-ivp} with $s=2$ and $n_1=n_2=1$. The unique solution of \eqref{eqn:DAEsystemQ} is clearly given by $x=0$ (the specification of the initial value $x_0$ is obsolete by $E=0$).\\
The above choice of $E$, $R$ and $Q$ leads to $M$ in \eqref{eq:Mop} being the zero operator, whereas $N$ as in \eqref{eq:Nop} corresponds to the multiplication with the skew-Hermitian matrix $-J$. Consequently, for $\lambda>0$, the iteration \eqref{alg:xcontr} now reads
\[
   x^{k+1}:=U x^{k},
   \quad 
   \text{where} \quad
   U=(I+\lambda J)(I-\lambda J)^{-1}.
\]
By invoking Lemma~\ref{lem:cayleyalg}, we see that $U$ is a~unitary matrix with $U\neq I_2$. Therefore, $\|x^{k+1}(t)\|=\|x^{k}(t)\|$ for all $k\in\mathbb{N}$, $t\in[0,T]$, and thus $\|x^{k+1}\|_{2}=\|x^{k}\|_{2}$. It can be concluded that the sequence $(x^k)$ does not converge in $L^2([0,T],\mathbb{R}^2)$ to $x=0$. A closer look to the iteration yields that $(x^k)$ does not even have a subsequence which weakly converges in $L^2([0,T],\mathbb{R}^2)$ to $x=0$.
\item Under Assumption~\ref{ass:overall-ivp}, \eqref{eq:rkdisscond} is equivalent to $E$ and $R$ fulfilling $\mathrm{rk}E=n$, or
$\mathrm{rk}\begin{bmatrix}
        E &
        R 
      \end{bmatrix}=n$ and $\alpha>0$.
\item In the case of ordinary differential equations, i.e., $E=I$, then Theorem~\ref{thm2} simplifies to the following:
\begin{enumerate}[(i)]
\item If $\mu>0$, or $\mathrm{rk}R=n$ and $\alpha <1$,
then $(z^k)$ converges to $z$ in $L^2([0,T],\mathbb R^n)$, and $(x^k)_k$ converges 
to $x$ in $H^1([0,T],\mathbb R^n)$.
\item $(x^k)_k$ converges 
to $x$ in $L^2([0,T],\mathbb R^n)$.  \hfill $\Box$
\end{enumerate}
\end{enumerate}
\end{remark}

The above discussion results in a Lions-Mercier-type dynamic iteration for PHS, which give summarize in Alg.~\ref{alg:mercier-PHS}.

\begin{algorithm}
\caption{Lions-Mercier-type Dynamic Iteration for PHS}\label{alg:mercier-PHS}
\fbox{\begin{minipage}{0.9\textwidth}
\begin{algorithmic}[1]
\Require problem data: $J_i$, $J_o$, $R_i$, $Q_i$, $E_i$, $B_i$, $u_i$, $x_{i,0}$ ($i=1,\dotsc, s$)
\Require choose for $i=1,\ldots, s$ $z^0_i \in L^2([0,T];\,\mathbb{R}^{n_i})$ 
        \qquad \% initial guess
\Require choose parameters $\alpha$, $\mu$, $\lambda$        
\State $R \, \leftarrow \diag(R_1,\dotsc,R_s)$, $E, Q$ analogously 
\State $K \, \leftarrow \, (1-\alpha)R+\mu EQ^{-1}-J_o $  
\State $z^0 \, \leftarrow \, \begin{bmatrix}(z_1^0)^\tr, \, \dotsc,\, (z_s^0)^\tr \end{bmatrix}^\tr$
\For{$k=0,1,2,\ldots$} 
    \For{$i=1,2,\dotsc,\, s$}
        \State solve for $x_i^k$
       \begin{align*}
         & \tfrac{d}{dt}E_ix_i^k=(J_i-\alpha R_i 
         +\mu E_iQ_i^{-1})Q_ix_i^k-\tfrac{1}{\lambda}Q_i x_i^k+  B_iu +  \tfrac{1}{\lambda}z_i^k
         \\
       &\text{with initial value } E_ix_i^k(0)=E_i x_{0,i}
       \end{align*}
    \EndFor    
    \State $x^k \, \leftarrow \,
          \begin{bmatrix}(x_1^k)^\tr, \, \dotsc,\, (x_s^k)^\tr \end{bmatrix}^\tr$
    \State $z^{k+1} \, \leftarrow \, 
         (I-\lambda  K)(I+\lambda  K)^{-1} 
        \left( 2 Qx^k-z^k \right)$ 
    \State partition $z^{k+1}$ into components $z_1^{k+1},\, \dotsc,\,  z_s^{k+1} $    
    \If{error small}
      return
    \EndIf  
\EndFor
\end{algorithmic}
\end{minipage}}
\end{algorithm}

\section{Numerical results}

We analyse the convergence rates for the new type of monotonic dynamic iterations (Lions-Mercier-type iteration). 
Then, we show numerical convergence results for the promising monotone convergence algorithms developed here. 
%

\subsection{Convergence rate discussion}

%

For $0<\alpha<1$ and
$K=((1-\alpha)R+\mu EQ^{-1}-J_o)$
the convergence rate is given by the contraction factor~\eqref{eq:estcayl3}
\begin{equation*}
  q^2
  = 1- \frac{4\lambda}{(1+\lambda \|K\|)^2 \|(K+J_o)^{-1}\|}.
\end{equation*}
The optimal $\lambda^*$ will have the smallest error reduction factor. We find:
\[
  \lambda^*(\mu) = \frac{1}{\| K \|} 
  \quad \; \text{and}\quad\;
    q^\ast\bigl(\mu,\, \lambda^{\ast}(\mu)\bigr)^2 
    = 1- \frac{1}{\|K\| \cdot \|(K+J_o)^{-1}\|} .
\]
\begin{remark}[Decoupled setting]
In the decoupled case, $J_o=0$, and hence the optimal error reduction reads:
\[
 q^\ast\bigl(\mu,\, \lambda^{\ast}(\mu)\bigr)^2 =1- \frac{1}{\cond{(K)}},
\]
which is only small for nearly equilibrated matrices $K$. This is in contrast to (e.g.) Jacobi dynamic iteration procedures, where we have immediate convergence. \hfill $\Box$
\end{remark}

Considering ODEs with $E=I$ and setting $\alpha=1$, we have $K=\mu Q^{-1} -J_o$.
\begin{remark}[ODE case]\label{rem:ODEcase}
\begin{enumerate}[a)]
    \item Transformation to $Q=I$: using the Cholesky factorization $Q=V^\top  V $, 
    we multiply~\eqref{eqn:DAEsystemQ}  with $V$, which yields
    \begin{align*}
    V \ddt x& %
                  =  V (J-R) (V^\top V)  x + V B u(t)
                \qquad
    \Leftrightarrow \quad \ddt \tilde{x}  = (\tilde{J} - \tilde{R}) \tilde{x} + \tilde{B} u(t)
    \end{align*}
    with $\tilde{x}=V x$, $\tilde J=V J V^\top, \tilde R= V R V^\top$ and $\tilde B= V B$ provides a remedy. 
    Correspondingly, $y= B^\top Q x$ is transformed into $\tilde y = {\tilde B}^\top \tilde x$.

    \item  \label{rem:ODEcase-decoupled}
    Decoupled setting ($J_o=0$): The reduction factor is  $q^\ast = 1-\tfrac{1}{\cond{(Q)}}$. Again, this is only small for scaling matrices $Q$ with small condition number. The above transformation to $Q=I$ yields $q^{*}=0$.
    An alternative is a shift with a matrix-valued $\lambda$ in Alg.~\ref{alg:mercier-PHS}.

    \item Coupled setting: for $Q=I$ and the Euclidean norm 
    this gives 
    \begin{align}\label{eq:lambda-opt-q-opt_Q=I}
    \lambda^*(\mu) & = \tfrac{1}{\sqrt{\mu^2+\lambda_{\max}(J_o^\tr J_o)}}
    \quad \text{and} \quad
       q^*\bigl(\mu,\lambda^*(\mu)\bigr)^2  =1-\tfrac{1}{\sqrt{1+{\lambda_{\max}(J_o^\tr J_o)}/{\mu^2}}}.
    \end{align}
\end{enumerate}

\end{remark}

\subsection{Results for monotone Lions-Mercier-type algorithm}

\subsubsection{Example with Jacobi failure}
For system~\eqref{eq:system-simple} %
holds 
$Q=I$ and $\lambda_{\max}(J_o^\tr J_o)=\nu^2$. This gives the optimal %
    contraction constant via~\eqref{eq:lambda-opt-q-opt_Q=I}: 
    $(q^*)^2 = {1-\tfrac{1}{\sqrt{1+ {\nu^2}/{\mu^2}}}}$.
We see the following:
\begin{remark}
\begin{enumerate}[a)]
    \item We have monotone convergence for the Lions-Mercier-type algorithm, Alg.~\ref{alg:mercier-PHS}. 
    \item In the limit of an undamped error bound $\omega \rightarrow 0^+$, the optimal convergence rate goes to one (for $\nu \neq 0$). 
     This indicates an extremely slow convergence. 
    \item One can get arbitrarily small convergence rates $q^*$ for $\mu \rightarrow \infty$. However, this corresponds to more and more damped error norms, which do not allow to estimate the numerically important undamped error. \hfill $\Box$
\end{enumerate}
\end{remark}

\subsubsection{Simple 2x2 system with scaling}
We introduce into the simple 2x2 example \eqref{eq:system-simple} %
a scaling matrix $Q$: 
\begin{equation}\label{eq:ivp-numericsQ}
    x' = (J\!-\!R) Q x, \quad\!  x(0)=x_0, \quad\! \text{with} \!\!
    \quad \!
  J\!-\!R = \begin{bmatrix}
  -\tau & \nu \\
  -\nu & -\tau
  \end{bmatrix}\!, \quad \!
  Q =\begin{bmatrix}
        q_1 & 0\\
        0   & q_2
  \end{bmatrix}\!, 
\end{equation}
where $\tau, \nu >0$ and  $q_1 \ge q_2 >0$.
In the decoupled case, $\nu=0$, the optimal reduction factor is given by (see Rem.~\ref{rem:ODEcase}\ref{rem:ODEcase-decoupled}):
\[
  q(\lambda^*) = 1- \tfrac{1}{\cond(Q)} 
  = 1 - \tfrac{q_2}{q_1}
\]
Fig.~\ref{fig:2x2Q} gives numerical results for the Lions-Mercier-type algorithm, showing that $Q$ not being a multiple of $I$ destroys the convergence in one step, although the system is fully decoupled.

\begin{figure}[!hbtp]
\centering
\begin{tabular}{@{}cc@{}}
    \includegraphics[width=0.482\textwidth]{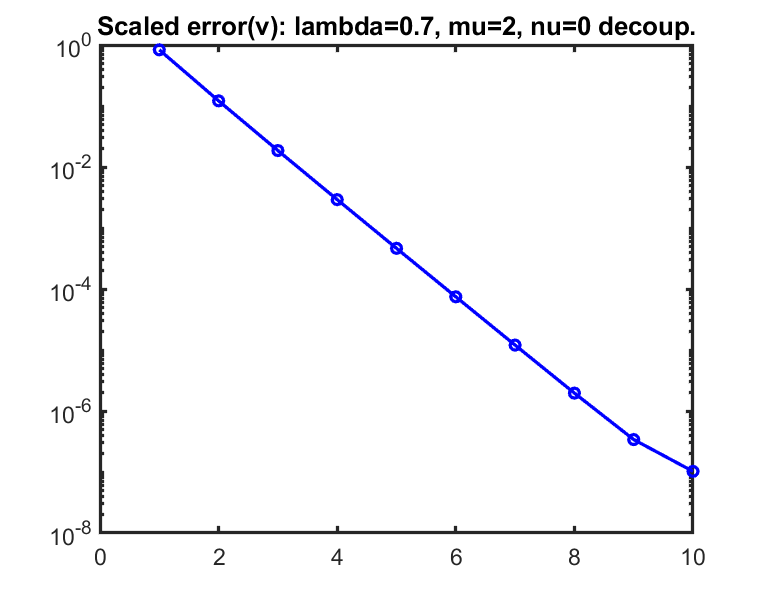}
    &
    \includegraphics[width=0.482\textwidth]{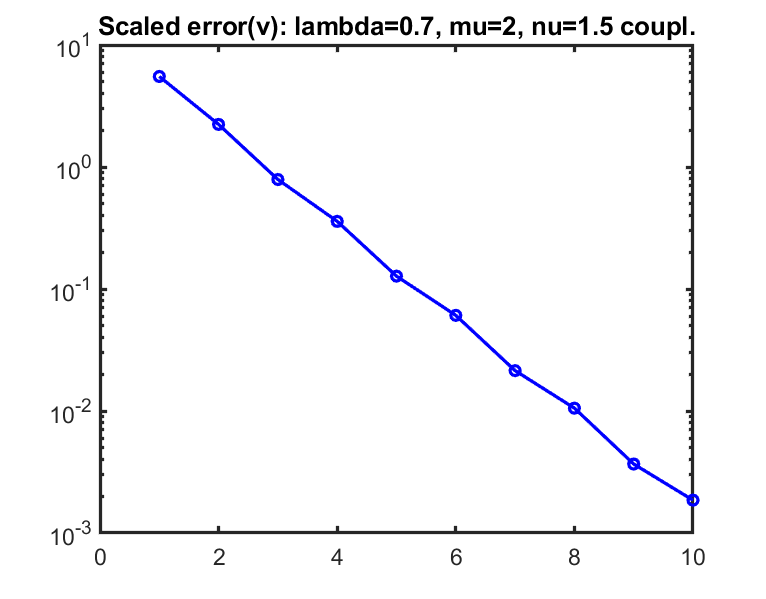}
\\[-1.75ex]
    {\small (a) decoupled, $Q\!=\!I$, non-optimal $\lambda$
     }
    &
    {\small (d) coupled, $Q\!=\!I$, non-optimal $\lambda$
      }
\\[2.5ex]
     \begin{minipage}[b]{0.40\textwidth}
        Decoupled with optimal $\lambda$:\\ the error reduced directly to zero 
        (no logarithmic plot). \\
        \mbox{} \\ \mbox{} \\ \mbox{}
      \end{minipage}
      &
     \includegraphics[width=0.482\textwidth]{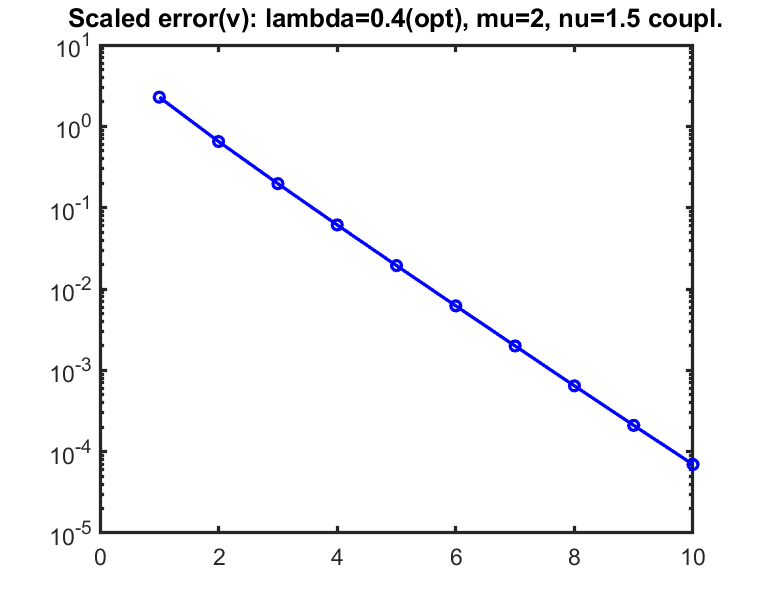}
\\[-1.75ex]
    {\small (b) decoupled, $Q=I$, optimal $\lambda$
    }
      &
    {\small (e) coupled, $Q=I$, optimal $\lambda$
      }
\\[2.5ex]
    \includegraphics[width=0.482\textwidth]{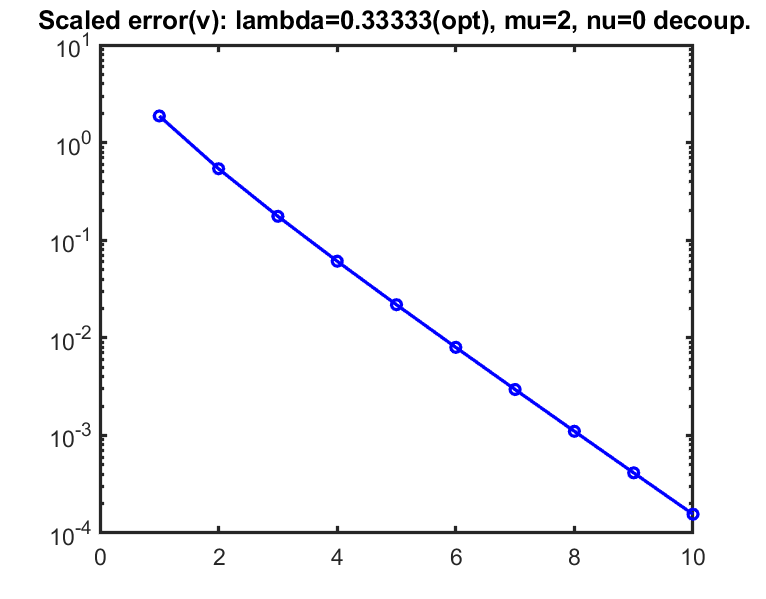}
    &
    \includegraphics[width=0.482\textwidth]{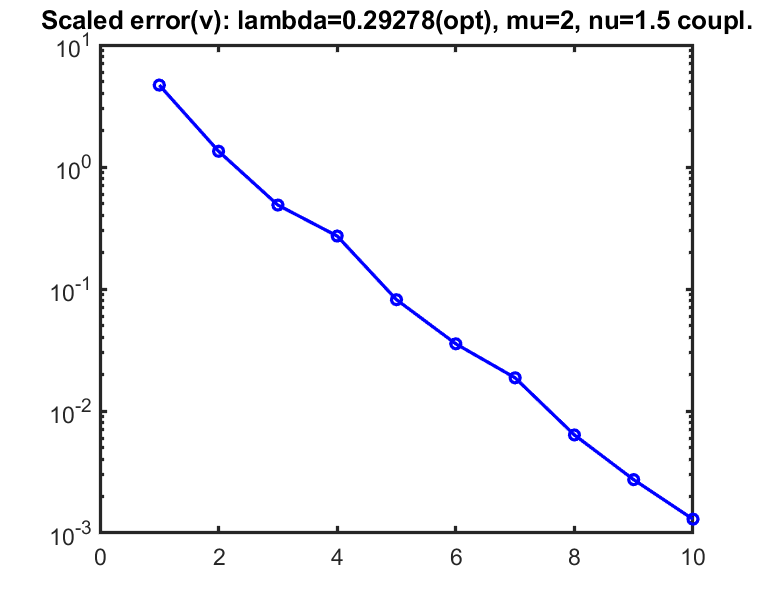}
\\[-1.75ex]
    {\small (c) decoupled, $Q\ne I$, optimal $\lambda$
    }
    &
    {\small (f) coupled, $Q\ne I$, optimal $\lambda$
     }
    \\
    $q_1=1.5$, $q_2=1$, $q(\lambda^*)\approx 0.3333$
    &
    $q_1=1.5$, $q_2=1$, $q(\lambda^*)\approx\frac{1}{3}$
     \end{tabular}
    
    \caption{Lions-Mercier-type iteration with~\eqref{eq:ivp-numericsQ}: number of iterations ($x$) versus scaled error in coupling term $v$/$z$ for various settings. First column provide the decoupled case ($\nu=0$), and the right column the coupled cases ($\nu=1.5$). Always $\tau=-0.5$ and $\omega=2.1$. }
    \label{fig:2x2Q}
\end{figure}

\begin{remark}
The convergence rates for the coupled cases -- even for the optimal choice of $\lambda$ should be improved to have more competitive dynamic iteration schemes. \hfill $\Box$
\end{remark}

\subsubsection{Two masses and three springs example with damping}
%
\begin{figure}[htb]
\begin{center}
\includegraphics[width=0.75\textwidth]{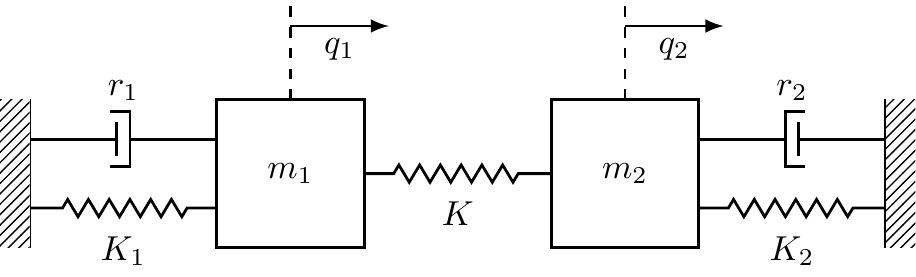}
\end{center}
\caption{\label{fig:two-masses_two-dampers_three-springs} ODE two masses oscillator with damping. The coordinates $q_1,\, q_2$ describe the position of the masses.}
\end{figure}

Now, we consider two masses $m_1,\, m_2 >0$, which are connected  via massless springs $K_1,\, K_2, K_3>0$ to walls with damping $r_1, r_2>0$, see  Fig.~\ref{fig:two-masses_two-dampers_three-springs}.
To set up the system, we have positions $q_1,\, q_2$ and momenta $p_1,\, p_2$ for the masses $m_1,\, m_2$, respectively. 
Then, the system can be modeled by the following Hamiltonian equation of motion:
\begin{equation*}%
\begin{aligned}
    \dot  q_{1} & = \tfrac{1}{m_{1}} p_{1}, 
    & \qquad
    \dot  q_{2} & = \tfrac{1}{m_{2}} p_{2}, 
    \\
     \dot p_{1} & %
                   =  -(K\!+\!K_1) q_{1} 
                    + K q_{2}\!-\! \tfrac{r_1}{m_1} p_1 ,
    & %
   \dot p_{2} & 
                 =  -(K\!+\! K_2 )  q_{2} 
                    + K q_1 \!-\! \tfrac{r_2}{m_2} p_2. 
\end{aligned}
\end{equation*}
Notice, this is not yet given in our PHS-format~\eqref{eqn:DAEsystemQ} for the  overall systems. Before we treat the overall system, we set up a description as coupled PH-subsystem as in Sec.~\ref{sec:coupled-ph-subsystems}, let $q=q_2$  denote the coupling variables, which we need for the first subsystem. 
Then, the two coupled PHS systems read:
    \begin{alignat*}{2}
        \dot x_1 & = (J_1-R_1) z_1 + B_1 u_1, 
        & \qquad\ \qquad \qquad \quad 
                    \dot x_2 & = (J_2-R_2) z_2 + B_2 u_2, \\
        y_1 & = B_1^\top z_1, 
        & 
                    y_2 & = B_2^\top z_2, \\
        & &          
                    u & = C y,
    \end{alignat*}
    with quantities for the first subsystem: 
    $x_1 =[p_1,\,q_1,\, q_1-q]^\tr$, 
    $ z_1 = Q_1  x_1$,  
    \begin{align*}
        &
        Q_1 \!=\!  \begin{bmatrix}
               \frac{1}{m_1} & 0 & 0 \\ 0 & K_1 & 0\\ 0 & 0 & K
           \end{bmatrix}\!, \;
           J_1 \!=\! \begin{bmatrix} 0 & -1 & -1 \\ 1 & 0 & 0 \\ 1 & 0 & 0 \end{bmatrix} \!, 
           \;
           R_1 \!=\! \begin{bmatrix}   r_1 
           & 0 & 0 \\ 
                                      0 & 0 & 0 \\ 
                                      0 & 0 & 0 \end{bmatrix} \!, 
           \;
           B_1 \!=\! \begin{bmatrix}
               0 \\ 0 \\ -1
          \end{bmatrix}
    \end{align*}
    and for the second subsystem we have: $x_2 =[p_2,\, q_2]^\tr$, $z_2= Q_2 x_2$     
    \begin{align*}
        & 
        Q_2 \!=\! \begin{bmatrix}
               \frac{1}{m_2} & 0 \\ 0 & K_2
        \end{bmatrix}\!, 
        \; J_2 \!=\! \begin{bmatrix} 0 & -1  \\ 1 & 0   \end{bmatrix}\!, 
        \quad
        \; R_2 \!=\! \begin{bmatrix}  r_2  
        & 0  \\ 0 & 0   \end{bmatrix}\!, 
        \quad
        B_2
        \!=\! \begin{bmatrix}
               1\\ 0 
        \end{bmatrix}\!, 
        \\
    \intertext{and the coupling interface reads:}
        u & \!=\! \begin{bmatrix}
               u_1 \\ u_2
        \end{bmatrix}, 
        \quad 
        y= \begin{bmatrix}
               y_1 \\ y_2
        \end{bmatrix}, 
        \quad 
        C= \begin{bmatrix}
               0 & 1 \\ -1 & 0  
        \end{bmatrix} .
    \end{align*}
Also the overall coupled system can be written in our pH-structure given in \eqref{eqn:DAEsystemQ},  with $F=\diag(F_1,F_2)$ for $F \in \{J,R,Q,B\}$, cf. Rem.~\ref{rem.coupledphs}:  %
\begin{align}
	\label{eq:2masses_ODE}
     \begin{bmatrix}
		x_1 \\ 
		x_2
	\end{bmatrix}'
	& = \left( [J-BCB^\top]-R \right) \cdot Q 
	  \begin{bmatrix}
		x_1 \\ 
		x_2
	\end{bmatrix}. %
\end{align}

 Numerical results for the Lions-Mercier-type Alg.~\ref{alg:mercier-PHS} are given in Fig.~\ref{fig:2masses}. %
 These results show that the estimate $q^\star \approx 0.975$ for the convergence rate (with an optimal choice leading to $q^\star \approx 0.944$) is quite pessimistic in this case (for $\mu=2$).

 \begin{figure}[!hbt]
     \centering
    \mbox{\!\!\!\!\includegraphics[width=0.36\textwidth]{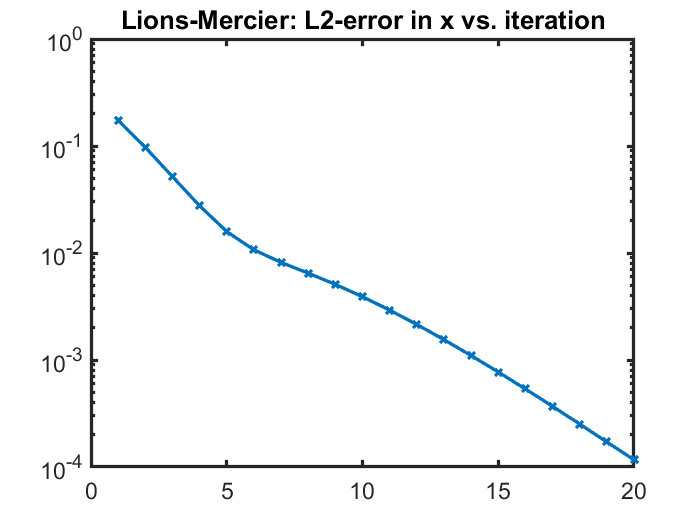}
     \!\!\!\!\!\!\!\includegraphics[width=0.36\textwidth]{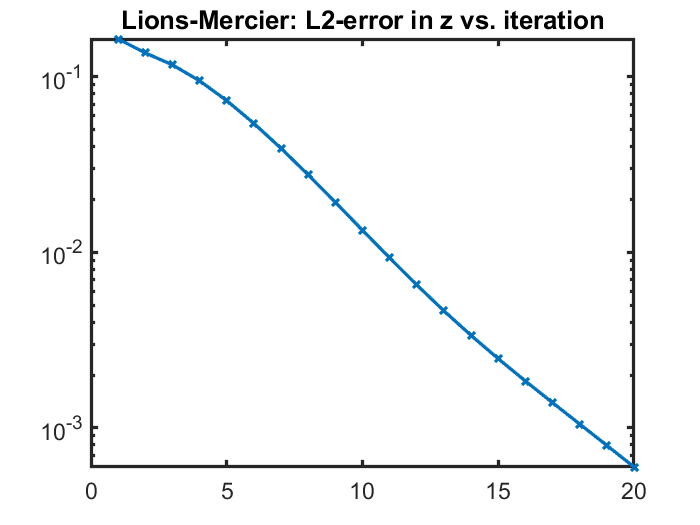}
     \!\!\!\!\!\!\!\includegraphics[width=0.36\textwidth]{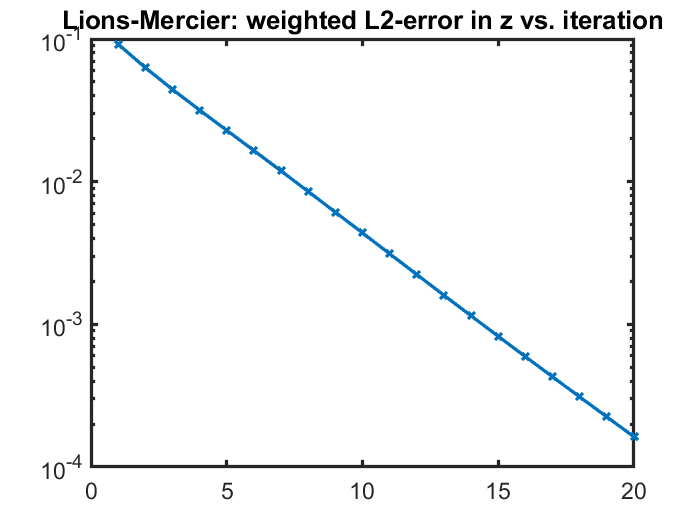}\!\!\!}
     \\
     \mbox{     \!\!\!\!\includegraphics[width=0.36\textwidth]{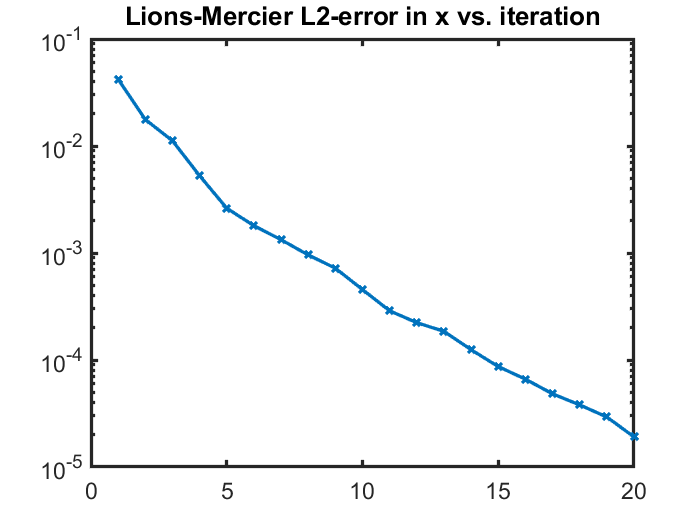}
     \!\!\!\!\!\!\!\includegraphics[width=0.36\textwidth]{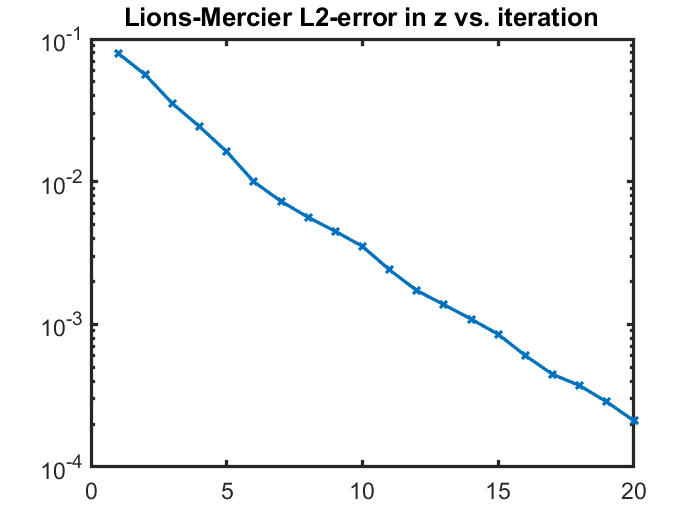}
     \!\!\!\!\!\!\!\includegraphics[width=0.36\textwidth]{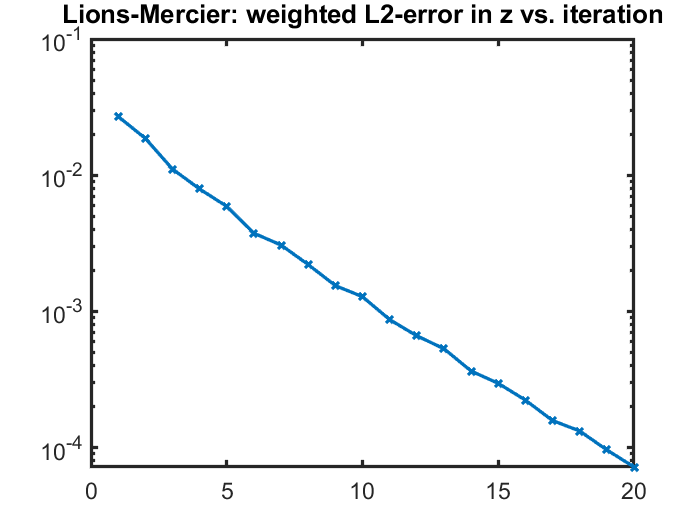}\!\!\!
    }
     \caption{Convergence analysis for Lions-Mercier-type iteration ($\lambda=1.5$, $\mu=2$ (first row)  $\mu=0.1$ (second row),  $\omega=2.2$, $\alpha=0.5$) for system \eqref{eq:2masses_ODE} with $K_i=2$, $m_i=2$ ($i=1,2$), $K=4$, $r_1=0.5$, $r_2=0.75$, initial values are zero apart from $p_{2,0}=0.1$. }
     \label{fig:2masses}
 \end{figure}
 \begin{remark}
Including coupling and output quantities as variables, i.e., 
$ \tilde x_i = \begin{bmatrix} x_i,\,  u_i,\,y_i \end{bmatrix}^\tr$ yields a DAE, which has no damping ($R$) in the coupling or output equations. Thus it does not fit into our framework.
$\hfill \Box$
 \end{remark}

Finally, we like to discuss a DAE example. 
To have a DAE multibody system, a simple choice are holonomic constraints. However, such a constraint comes without  dissipation. Hence $\text{rk}(E,R)$ is not full. Therefore, we change the field of applications. 
\subsection{Circuit example (DAE)}

The electric circuit in Fig.~\ref{fig:electric_circuit} can be modelled as follows:
for $x^\tr = \begin{bmatrix}
u_1, \, u_2,\, u_3,\, \jmath_1,\, u_4,\, u_5 \end{bmatrix}$ and
$E=\text{diag}(0, C_1, 0, L_1, 0, C_2)$
$B^\tr= \begin{bmatrix} 1,\, 0,\, 0,\, 0,\, 0,\, 0\end{bmatrix}$
\begin{align*}
&(E x)' = B \imath(t) +\\
&\left(\! \begin{bmatrix} \begin{array}{cccc|cc}
	0 &&&&&\\
	& 0 &&&&\\
	&   & 0 &  1&&\\
	&   &-1 &  0& 1 &\\ \hline
	&   &   &-1 &   0 &\\
	&   &   &   &   &  0 
	\end{array}
\end{bmatrix}
\!-\!
\begin{bmatrix}
\begin{array}{cccc|cc}
	\tfrac{1}{R_1} & -\tfrac{1}{R_1} &&&&
	\\
	-\tfrac{1}{R_1} & \tfrac{1}{R_1}\!+\!\tfrac{1}{R_2} & -\tfrac{1}{R_2} &&&
	\\
	& -\tfrac{1}{R_2} &\tfrac{1}{R_2}+\tfrac{1}{R_3}  & &  &	\\
	& & & 0 &&\\ \hline
	&       &  & & \tfrac{1}{R_4}  & -\tfrac{1}{R_4} \\
	&       &  & &-\tfrac{1}{R_4} & \tfrac{1}{R_4} \!+\! \tfrac{1}{R_5}
	\end{array}
\end{bmatrix}
\!\right) x.
\end{align*}
\begin{figure}[!htb]
\includegraphics[width=\textwidth]{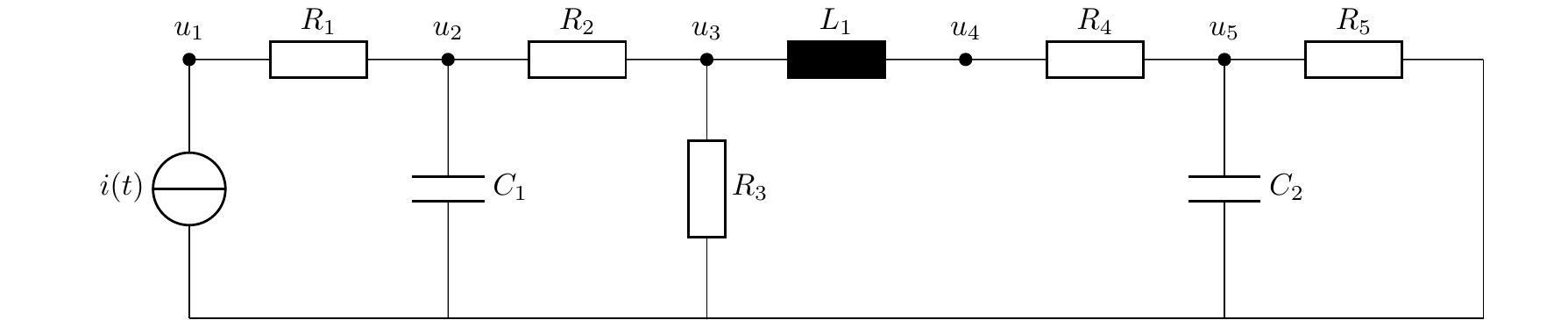}
\caption{Example of a simple electric circuit; $R_i>0$, $L_i>0$, $C_i>0$ and given $\imath\in L^{2}[0,T]$.}
\label{fig:electric_circuit}
\end{figure}
This example fulfills the rank condition $\text{rk}(E,R) = n=6$ and has block-diagonal $R$.

Numerical results for the Lions-Mercier-type Alg.~1 are given in Fig.~\ref{fig:electric_simple_circuit_results}. These
results show that the estimate $q \approx 0.9998$ for the convergence rate (with an optimal choice leading to $q^\star \approx 0.9980$) is quite pessimistic in this case, too.

\begin{figure}
    \centering
    \mbox{
    \!\!\!\!\includegraphics[width=0.36\textwidth]{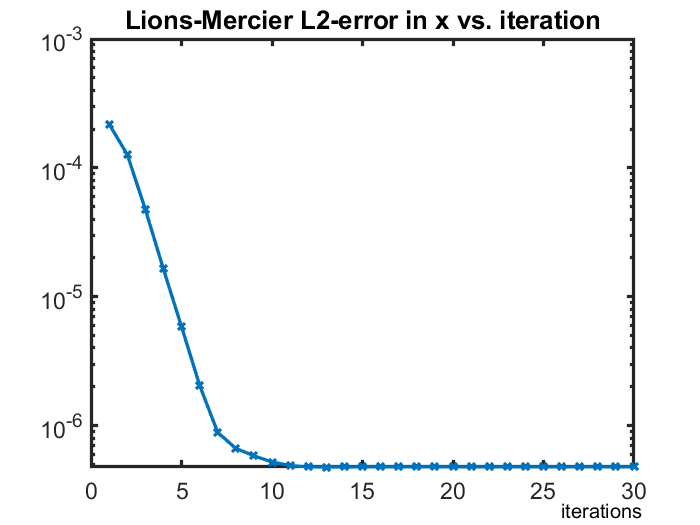}
     \!\!\!\!\!\!\includegraphics[width=0.36\textwidth]{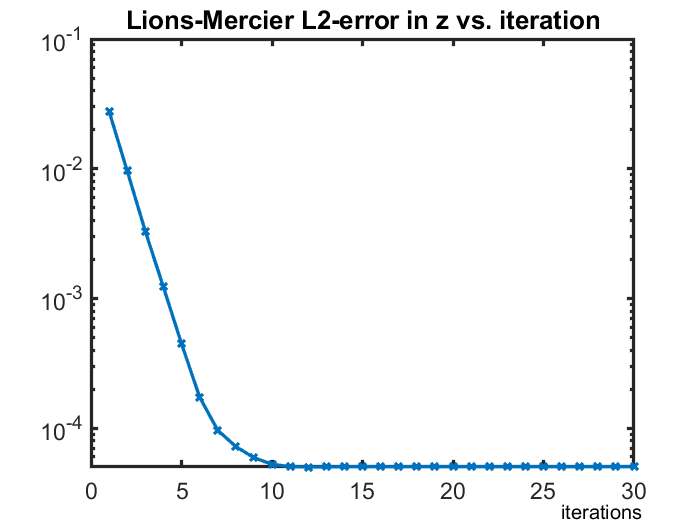}
     \!\!\!\!\!\!\includegraphics[width=0.36\textwidth]{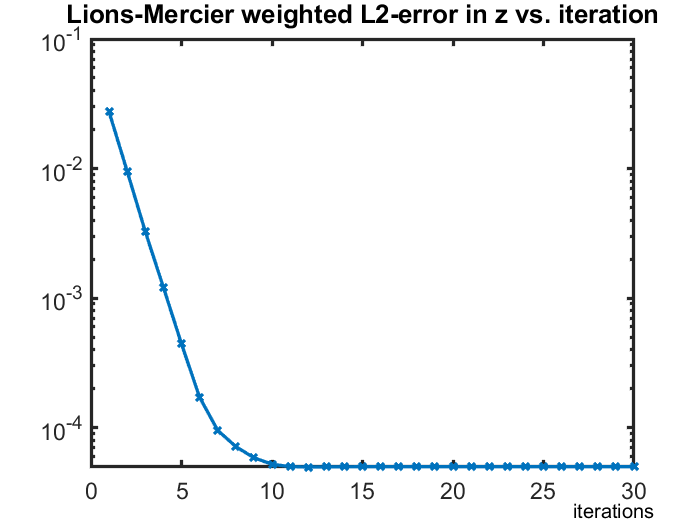}\!\!\!\!
    }
     \\
     {\small (a) error reduction: achieved precision vs. iteration count \\(left in $x$, center in $z$, right in $z$ with weighted $L^2$-norm);}
     \\[0.5ex]
     \mbox{
     \!\!\!\!\includegraphics[width=0.27\textwidth]{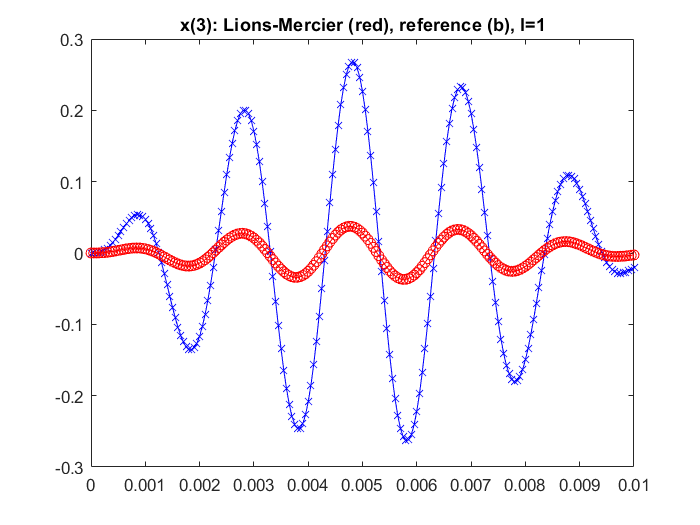}
     \!\!\!\!\!\!\includegraphics[width=0.27\textwidth]{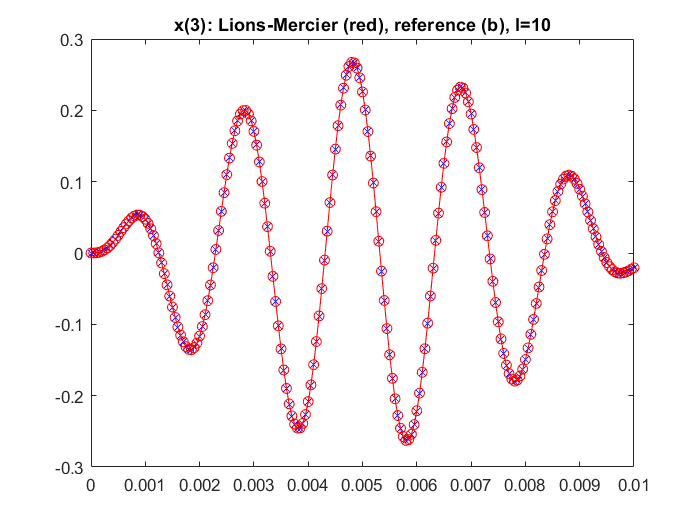}
     \!\!\!\!\!\!\includegraphics[width=0.27\textwidth]{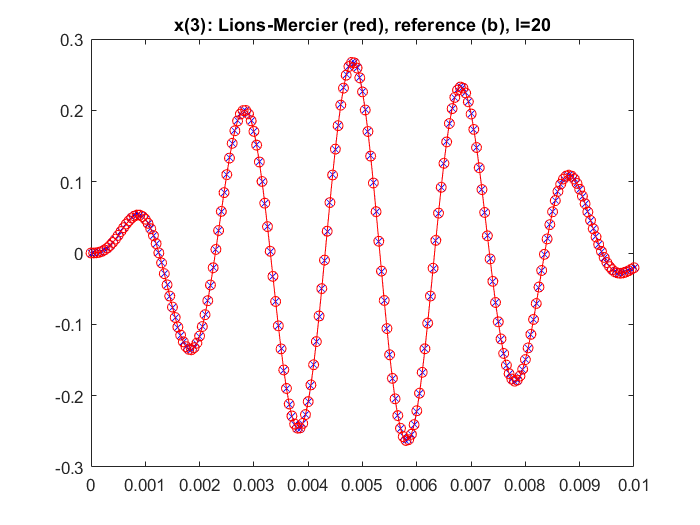}
     \!\!\!\!\!\!\includegraphics[width=0.27\textwidth]{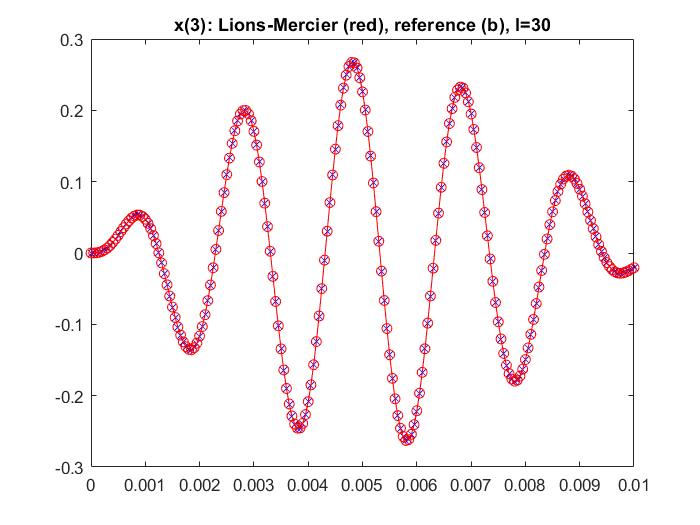}
     \!\!\!\!
    }
     \\
     {\small (b) convergence in $x_3$: we give results for iteration $l=1,10,20,30$;}
          \\[0.5ex]
          \mbox{%
     \!\!\!\!\includegraphics[width=0.27\textwidth]{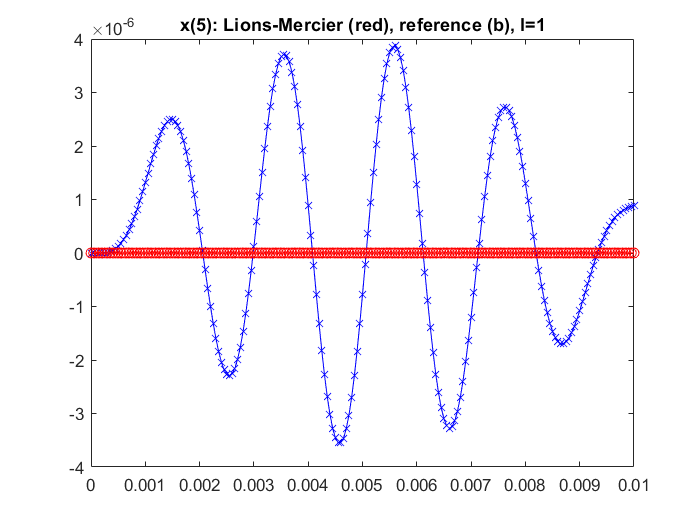}
     \!\!\!\!\!\!\includegraphics[width=0.27\textwidth]{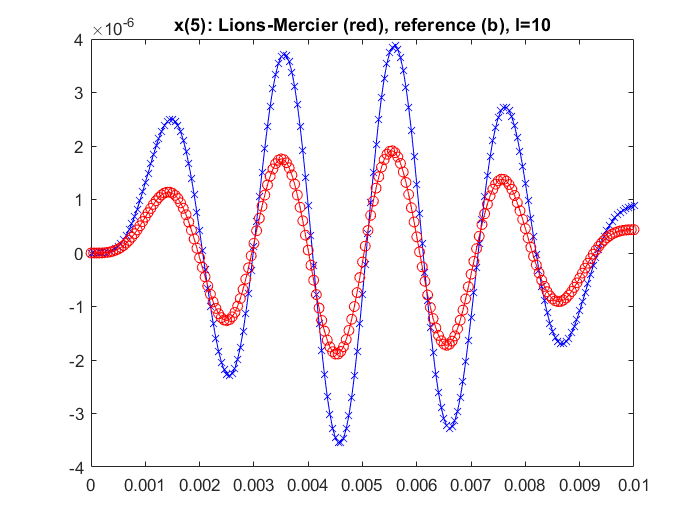}
     \!\!\!\!\!\!\includegraphics[width=0.27\textwidth]{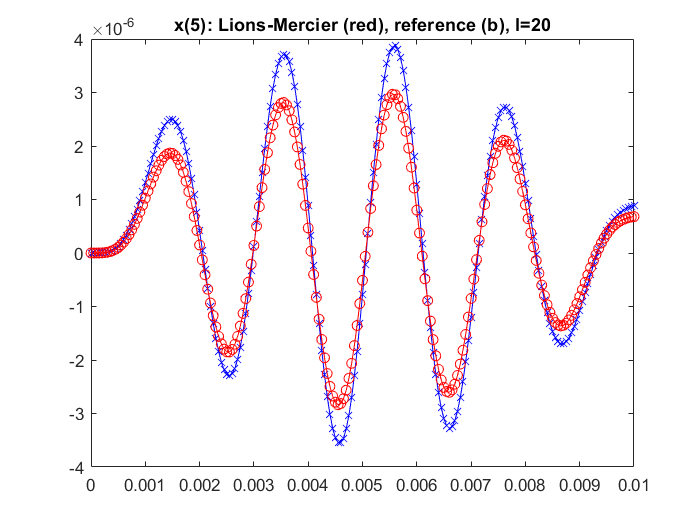}
     \!\!\!\!\!\!\includegraphics[width=0.27\textwidth]{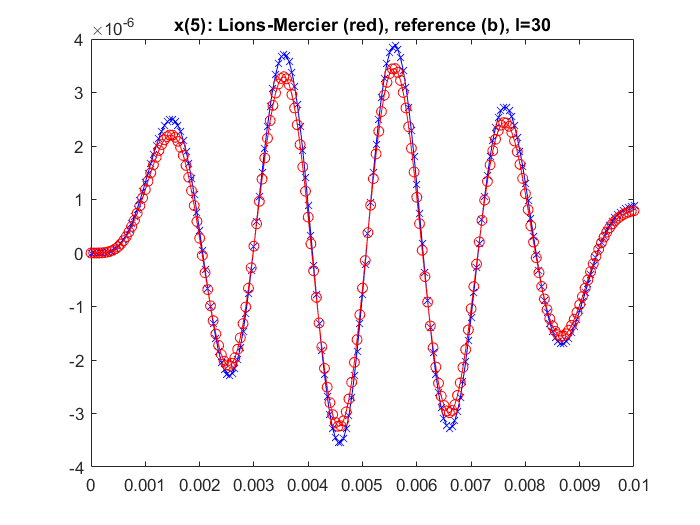}
     \!\!\!\!
    }
     \\
     {\small (c) convergence in $x_5$: we give results for iteration $l=1,10,20,30$.}

    \caption{Results for the academic circuit Fig.~\ref{fig:electric_circuit} using:
    $R_1=R_2=R_3=R_4=0.5,\, R_5=5$, $C_1=C_2=5\cdot 10^{-4}$, $L=20$, $i(t)=\sin(2\pi\cdot 50 t)\cdot \sin(2\pi\cdot 500 t)$.   
    Lions-Mercier parameters: $\lambda=1.2$, $\mu=1$,  $\omega=2.2$, $\alpha=0.2$. (a) shows the $L^2$ error reduction for different variables, (b) and (c) depict the convergence in a variable of the first subsystem $x_3$ and in the second subsystem $x_5$, respectively.}
    \label{fig:electric_simple_circuit_results}
\end{figure}

\begin{remark}[Rank condition for circuits]
We discuss circuit equation modelled by modified nodal analysis. \\
i) Let a circuit include one voltage sources. Then the model has a pure algebraic equation, which has no dissipation. Thus the rank condition is not satisfied.
\\ 
ii) Electric circuits with current sources, capacitors, inductors and resistors can be made to fulfill the rank condition by adding sufficiently many resistors.
\\
iii) The coupling---without auxiliary variables---can be facilitated via inductors. 
\hfill $\Box$
\end{remark}

\section{Conclusions}

For a general PHS setting, we have shown existence and uniqueness of weak solutions. Furthermore, the perspective of coupled systems and coupling variables are treated. Then, it is demonstrated that Jacobi-type of dynamic iteration may fail in practice due to finite precision. To circumvent this, a Mercier-Lions-type dynamic iteration scheme is developed, which guarantees monotone convergence in a related variable. The proof is based on properties of the Cayley transform. Numerical results demonstrate the monotone convergence. As a drawback, we observe that the currently achieved optimal convergence rates still need some improvements in order to pave the way for a broader application.


\bibliography{references}
\bibliographystyle{plain}

\end{document}